\documentclass{amsart}

\usepackage{amscd,amsfonts,mathrsfs,amsthm,amssymb}
\usepackage{comment}
\usepackage{mathtools}
\usepackage{enumerate}
\usepackage{multicol}        			
\usepackage{color}           			
\usepackage{array}
\usepackage{ae}
\usepackage{accents}
\usepackage{epsfig}
\usepackage[e]{esvect}

\usepackage{empheq}
\usepackage[latin1]{inputenc}

\usepackage{hyperref}
\usepackage[alphabetic, msc-links, backrefs]{amsrefs}

\usepackage{calrsfs}
\DeclareMathAlphabet{\pazocal}{OMS}{zplm}{m}{n}

\newcommand{\R}{\mathbf{R}}

  \newcommand{\Ss}{\mathcal S}
 
 \newcommand{\RR}{\mathbf{R}}  
 \newcommand{\BB}{\mathbf{B}}  
 \renewcommand{\SS}{\mathbf{S}}  

    \newcommand{\dist}{\operatorname{dist}}

 \newcommand{\eps}{\epsilon}
 \newcommand{\Tan}{\operatorname{Tan}}

\newcommand{\ee}{\mathbf e}

\newcommand{\pdf}[2]{\frac{\partial #1}{\partial #2}}

\newtheorem*{theorem*}{Theorem}
\newtheorem{theorem}{Theorem}[section]

\newtheorem{lemma}[theorem]{Lemma}

\newtheorem{corollary}[theorem]{Corollary}
\newtheorem{proposition}[theorem]{Proposition}

\theoremstyle{definition}
\newtheorem{remark}[theorem]{Remark}

\def\pproof#1{\@ifnextchar[\opargproof
{\opargproof[\it Proof of #1.]}}
\def\opargproof[#1]{\par\noindent {\bf #1 }}

\makeatother

\numberwithin{equation}{section}

\begin{document}

\title[Graphical translators]{Graphical Translators for Mean Curvature Flow}
\author[D. Hoffman]{\textsc{D. Hoffman}}

\address{David Hoffman\newline
 Department of Mathematics\newline
 Stanford University \newline
   Stanford, CA 94305, USA\newline
{\sl E-mail address:} {\bf dhoffman@stanford.edu}}

\author[T. Ilmanen]{\textsc{T. Ilmanen}}

\address{Tom Ilmanen\newline
Department of Mathematics \newline
E. T. H. Zürich \newline 
Rämistrasse 101, 
8092 Zürich,
Switzerland\newline
{\sl E-mail address:} {\bf tom.ilmanen@math.ethz.ch}
}

\author[F. Martin]{\textsc{F. Martín}}

\address{Francisco Martín\newline
Departmento de Geometría y Topología  \newline
Instituto de Matemáticas IE-Math Granada \newline
Universidad de Granada\newline
18071 Granada, Spain\newline
{\sl E-mail address:} {\bf fmartin@ugr.es}
}
\author[B. White]{\textsc{B. White}}

\address{Brian White\newline
Department of Mathematics \newline
 Stanford University \newline 
  Stanford, CA 94305, USA\newline
{\sl E-mail address:} {\bf bcwhite@stanford.edu}
}

\date{July 5, 2018.  Revised March 23, 2019}
\subjclass[2010]{Primary 53C44, 53C21, 53C42}
\keywords{Mean curvature flow, singularities, monotonicity formula, area estimates, comparison principle.}
\thanks{F. Martín is partially supported by the MINECO/FEDER grant MTM2017-89677-P and  by the
Leverhulme Trust grant IN-2016-019.
B. White was partially supported by grants from the Simons Foundation
(\#396369) and from the National Science Foundation (DMS~1404282, DMS~ 1711293).}

\begin{abstract}
In this paper we provide a full classification of complete translating graphs in $\RR^3$.
We also construct  $(n-1)$-parameter families of new examples of translating graphs in $\RR^{n+1}$.
\end{abstract}

\maketitle

\setcounter{tocdepth}{1}
  \tableofcontents
\section{Introduction}

\label{sec:intro}

A {\bf translator} is a hypersurface $M$ in $\R^{n+1}$ such that 
\[
   t\mapsto M- t \,\ee_{n+1}
\]
is a mean curvature flow, i.e., such that normal component of the velocity at each point is
equal to the mean curvature at that point:
\begin{equation}\label{general-translator-equation}
   \overrightarrow{H} = -\ee_{n+1}^\perp.
\end{equation}
If a translator $M$ is the graph of function $u:\Omega\subset\RR^n\to\RR$, we will
say that $M$ is a {\bf translating graph}; in that case, we also refer to the
function $u$ as a translator, and we say that $u$ is complete if its graph
is a complete submanifold of $\RR^{n+1}$.  
Thus $u:\Omega\subset\RR^n\to\RR$ is a translator if and only if
it solves the 
translator equation (the nonparametric form of~\eqref{general-translator-equation}):
\begin{equation}\label{divergence-translator-equation}
  D_i\left( \frac{D_iu}{\sqrt{1+ |Du|^2}} \right) = -\frac1{\sqrt{1+|Du|^2}}.
\end{equation}
The equation can also be written as
\begin{equation}\label{translator-equation}
   (1 + |Du|^2) \Delta u - D_iu\,D_ju\,D_{ij}u + |Du|^2 +1 = 0.
\end{equation}

In this paper, we classify all complete translating graphs in $\RR^3$.
In~\cite{scherkon}, we construct a family of translators that includes analogs of 
the classical Scherk doubly periodic minimal surfaces and of helicoids.
In another paper~\cite{annuli}, we construct a two-parameter family of translating annuli.
See~\cite{survey} for a survey of various results about translators.

Before stating our classification theorem, we recall the known examples
of translating graphs in $\RR^3$.   First, the Cartesian product of the grim reaper curve with $\RR$
is a translator:
\newcommand{\grim}{\mathcal{G}}
\begin{align*}
   &\grim: \RR \times (-\pi/2,\pi/2) \to \RR, \\
   &\grim(x,y) =  \log(\cos y).
\end{align*}
We refer to it as the {\bf grim reaper surface}~(Fig.~\ref{default}).
\begin{figure}[htbp]
\begin{center}
\includegraphics[width=.78\textwidth]{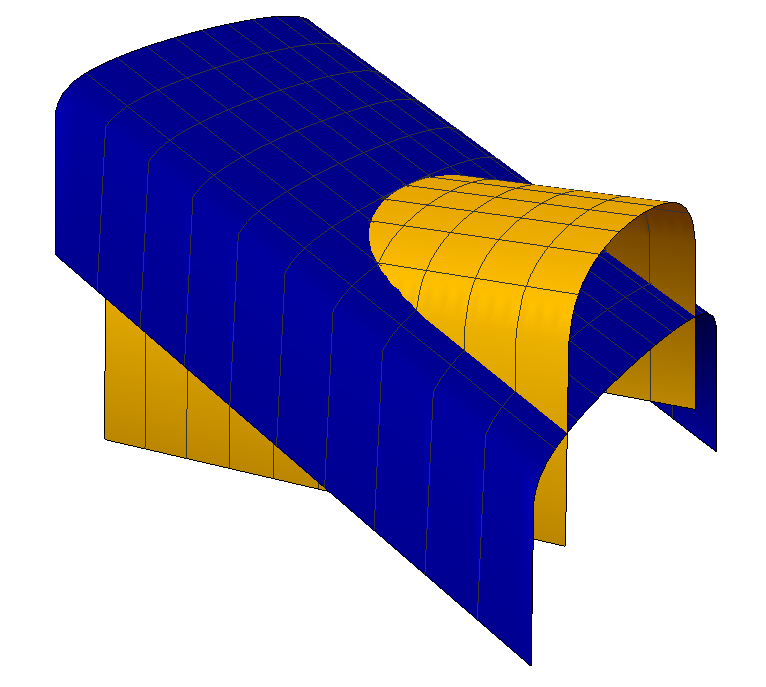}
\caption{The  grim reaper surface in $\R^3$, and that surface tilted by angle $\theta=-\pi/4$
and dilated by $1/\cos(\pi/4)$.}
\label{default}
\end{center}
\end{figure}

Second, if we rotate the grim reaper surface by an angle $\theta\in (0,\pi/2)$ about the $y$-axis and dilate
by $1/\cos\theta$, the resulting surface is again a translator, given by
\begin{equation}\label{tilted-grim-reaper-equation}
\begin{aligned}
&\grim_\theta:  \RR  \times(-b,b)\to  \RR,   \\
&\grim_\theta(x,y) = 
 \frac{\log (\cos (y \cos \theta ))}{\cos^2\theta}+x \tan\theta,
\end{aligned}
\end{equation}
where $b=\pi/(2\cos\theta)$.  
Note that as $\theta$ goes from $0$ to $\pi/2$, the width $2b$ of the strip goes
from $\pi$ to $\infty$.
We refer to these examples as {\bf tilted grim reaper surfaces}.

Every translator $\RR^3$ with zero Gauss curvature is (up to translations and 
up to rotations  about a vertical axis) a grim reaper surface, a tilted grim reaper surface, or a vertical plane.
See~\cite{mss} or Theorem~\ref{convexity-theorem} below.

In \cite{CSS}, J. Clutterbuck, O. Schn\"urer and F. Schulze  (see also \cite{Altschuler-Wu}) 
proved for each $n\ge 2$ that there is a unique (up to vertical translation) entire, rotationally invariant
function $u: \RR^n\to \RR$ whose graph is a translator.
It is called the {\bf bowl soliton}~(Fig.~\ref{bowl}).
\begin{figure}[htbp]
\begin{center}
\includegraphics[width=.75\textwidth]{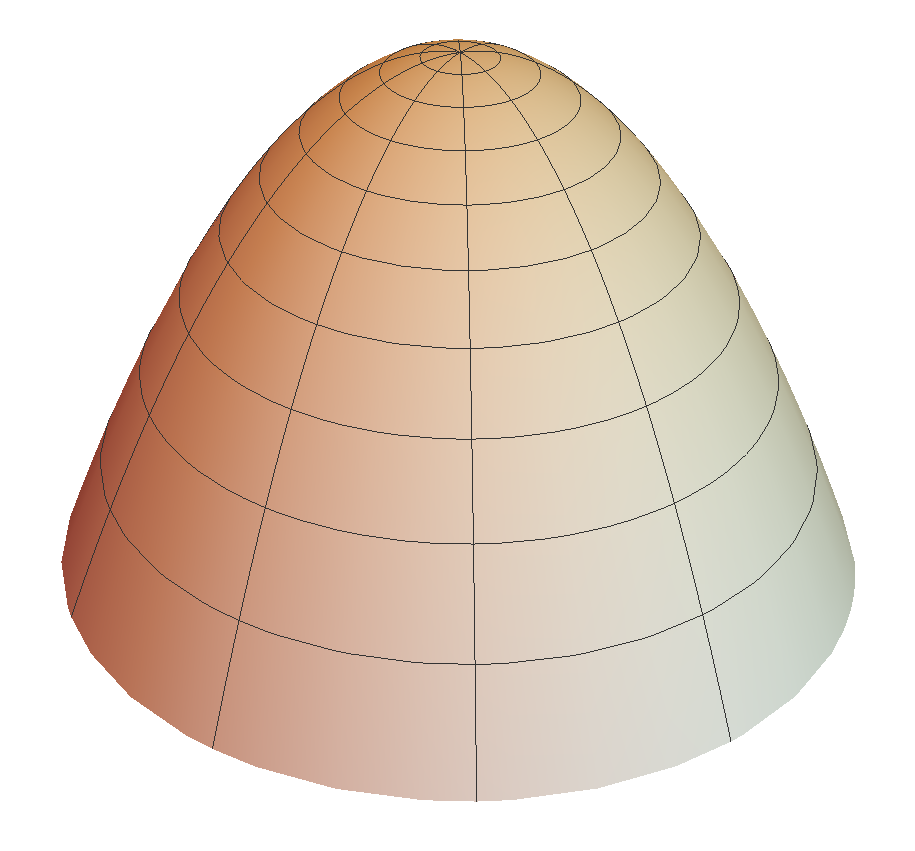}
\caption{The bowl soliton.  As one moves down, the slope tends to infinity, 
and thus the end is asymptotically cylindrical.}
\label{bowl}
\end{center}
\end{figure}

In addition to the examples described above, Ilmanen (in unpublished work) proved that for
each $0< k < 1/2$, there is a translator 
$
   u: \Omega\to \RR
$
with the following properties:
 $u(x,y)\equiv u(-x,y)\equiv u(x,-y)$, 
$u$ attains its maximum at $(0,0)\in \Omega$, and
\[
   D^2u(0,0)  
   =\begin{bmatrix} -k &0 \\ 0 &-(1-k) \end{bmatrix}.
\]
The domain $\Omega$ is either a strip $\RR\times(-b,b)$ or $\RR^2$.
He referred to these examples as {\bf $\Delta$-wings} (Fig.~\ref{intro}.)
As $k\to 0$, he showed that the examples converge to the grim reaper surface.
Uniqueness (for a given $k$) was not known.   It was also not known
which strips $\RR\times (-b,b)$ occur as domains of such examples.
This paper is primarily about translators in $\RR^3$, but in Section~\ref{higher-delta-wings}
 we extend Ilmanen's original proof to get $\Delta$-wings in $\RR^{n+1}$ that have 
  prescribed principal curvatures at the origin.
  For $n\ge 3$, the examples include entire graphs that are not rotationally invariant.
 In Section~\ref{higher-slabs-section}, 
 we modify the construction to produce a family of $\Delta$-wings in $\RR^{n+2}$
 over any given slab of width $>\pi$.  
 {See~\cite{wang} for a different construction
 of some higher dimensional graphical translators. 
 } 

\begin{figure}[htbp]
\begin{center}
\includegraphics[width=.75\textwidth]{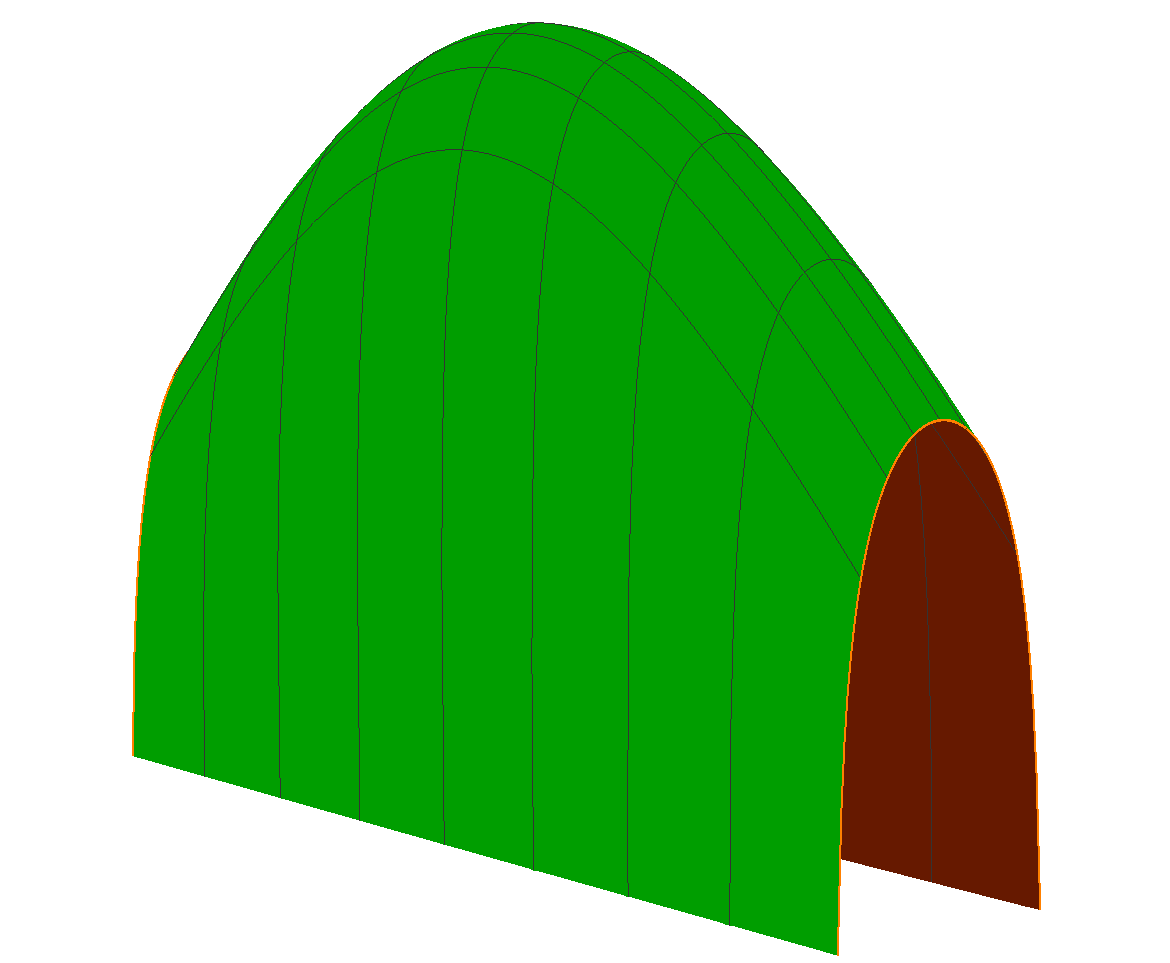}
\caption{The $\Delta$-wing of width $\sqrt{2} \, \pi$ given by Theorem~\ref{main-theorem}.
 As $y \to\pm \infty,$ this $\Delta$-wing
is asymptotic to the tilted grim reapers $\grim_{-\frac{\pi}{4}}$ and $\grim_{\frac{\pi}{4}}$,
respectively.}
\label{intro}
\end{center}
\end{figure}

The main result of this paper is the following theorem.

\begin{theorem}\label{main-theorem}
For every $b>\pi/2$, there is (up to translation) a unique complete, strictly convex
translator 
$
   u^b: \RR\times (-b,b)\to\RR.
$
Up to isometries of $\RR^2$, the only other complete translating graphs in $\RR^3$
are the grim reaper surface, the tilted grim reaper surfaces, and the bowl soliton.
\end{theorem}

We now describe previously known classification results.
Spruck and Xiao recently proved the very powerful theorem 
that every  translating graph in $\RR^3$ is convex~\cite{spruck-xiao}*{Theorem~1.1}.
(See~\cite{survey}*{\S6} for a somewhat simplified presentation of their proof.)
Thus it suffices to classify convex examples.
L.~Shahriyari \cite{shari} proved that if $u: \Omega\subset \RR^2\to \RR$ is a complete
translator, then $\Omega$ is (up to rigid motion) one of the following:
the plane $\RR^2$, a halfplane, or 
a strip $ \RR \times (-b,b)$ with $b\ge \pi/2$.

In~\cite{wang}, X. J. Wang proved that the only entire convex translating graph  is the
bowl soliton, and that there are no complete convex translating graphs defined over
halfplanes.  Thus by the Spruck-Xiao Convexity Theorem, 
the bowl soliton is the only complete translating graph defined over a plane or halfplane.

It remained to classify the translators $u:\Omega\to \RR$ whose domains
are strips.
The main new contributions in this paper are:
\begin{enumerate}
\item\label{existence-item} For each $b>\pi/2$, we prove (Theorem~\ref{unique-Delta-wings-theorem}) 
existence and uniqueness (up to translation)
of a complete translator $u^b: \RR \times (-b,b) \to \RR$ that is not a tilted grim reaper.
\item We give a simpler proof (see Theorem~\ref{halfplane-theorem}) 
that there are no complete graphical translators in $\R^3$ defined over halfplanes in $\R^2$.
\end{enumerate}

We remark that Bourni, Langford, and Tinaglia have recently given a different proof of the existence
(but not uniqueness) in~\eqref{existence-item} \cite{bourni-et-al}.   

\section{Preliminaries}

Here we gather the main 
properties of translators
that will be used in this paper.

As observed by Ilmanen~\cite{ilmanen},  a {hyper}surface $M\subset\RR^{n+1}$ 
is a translator if and only if
it is minimal with respect to the Riemannian metric
\[
    g_{ij}(x_1, \ldots,x_{n+1}) = \exp\left(-\frac2n x_{n+1} \right) \delta_{ij}.
\]
Thus we can freely use curvature estimates and compactness theorems from minimal surface theory;
cf.~\cite{white-intro}*{chapter 3}.
In particular, if $M$ is a graphical translator, then (since vertical translates of it are also $g$-minimal)
$\left<\ee_3,\nu\right>$ is a nowhere vanishing Jacobi field, so $M$ is a stable $g$-minimal surface.
It follows that any sequence $M_i$ of complete translating graphs in $\RR^3$ has a subsequence that
converges smoothly to a translator $M$.  Also, if a translator $M$ is the graph of
a function $u:\Omega\to\RR$, then $M$ and its vertical translates from a $g$-minimal
foliation of $\Omega\times\RR$, from which it follows that $M$ is $g$-area minimizing in $\Omega\times\RR$,
and thus that if $K\subset \Omega\times \RR$ is compact, then the $g$-area of $M\cap K$ is
at most $1/2$ of the $g$-area of $\partial K$.  In this paper, we will consider various sequences
of translators that are manifolds-with-boundary.  In the situations we consider, the area bounds described above
together with standard compactness theorems for minimal surfaces (such as those
in~\cite{white-curvature-estimates}) give smooth, subsequential convergence,
including at the boundary.  
(The local area bounds and bounded topology mean that the only boundary singularities
that could arise would be boundary branch points.   In the situations that occur in this paper, 
obvious barriers preclude boundary branch points.)

The situation for higher dimensional translating graphs is more subtle; see Section~\ref{appendix}.

\begin{theorem}\label{vertical-plane-theorem}
Let $M\subset\RR^3$ be a smooth, connected translator with nonnegative mean curvature.
If the mean curvature vanishes anywhere,  then $M$ is contained
in a vertical plane.
\end{theorem}

\begin{proof}
As mentioned above, the mean curvature $\left<\nu, \ee_3\right>$ is a Jacobi field.
By hypothesis, it is nonnegative.
By the strong maximum principle, if it vanishes anywhere,
it vanishes everywhere, so that $M$ is contained in $\Gamma\times\RR$ for some curve in $\RR^2$.
The result follows immediately.
\end{proof}

\begin{theorem}\label{convexity-theorem}
Let $M\subset \RR^3$ be a complete translator with positive mean curvature.
Then it is a graph and the Gauss curvature is everywhere nonnegative.
 If the Gauss curvature vanishes anywhere,
then it vanishes everywhere and $M$ is a grim reaper surface or tilted grim reaper
surface.
\end{theorem}

\begin{proof}
Nonnegativity of the Gauss curvature is the main result of~\cite{spruck-xiao}.
If the curvature vanishes anywhere, it vanishes everywhere because
$\kappa_1/H$ satisfies a strong maximum principle (where $0\le\kappa_1\le\kappa_2$ are
the principle curvatures).  See for example~\cite{white-nature}*{Theorem~3}.

The last assertion follows from work of Martin, Savas-Halilaj, and Smoczyk \cite{mss}*{Theorem~B}.
We can also prove it directly as follows.
Suppose that $M$ is a translator with Gauss curvature $0$.
By elementary differential geometry, $M$ is a ruled surface and the Gauss map is constant along
the straight lines.  Consequently if $L$ is a line in $M$, we can find
a grim reaper surface (tilted unless $L$ is horizontal) such that $\Sigma$ is 
tangent to $M$ along $L$.  By Cauchy-Kowalevski, $M=\Sigma$.
\end{proof}

\begin{corollary}\label{divergent-sequence-corollary}
Suppose that $M\subset \RR^3$ is a complete graphical translator.   If $p_i$ is a divergent
sequence in $M$, then $M-p_i$ converges smoothly (after passing
to a subsequence) to a vertical plane, a grim reaper surface,
or a tilted grim reaper surface.
\end{corollary}

\begin{proof} 
If $M$ is not strictly convex, then by Theorem~\ref{convexity-theorem} it is a grim reaper surface
or tilted grim reaper surface, and the corollary is trivially true.

Thus suppose that $M$ is strictly convex.  Then the Gauss map maps
$M$ diffeomorphically to an open subset of the upper hemisphere of $S^2$.
Since $M-p_i$ is a stable minimal surface with respect to the Ilmanen metric,
a subsequence will converge smoothly to a complete translator $M'$.  The Gauss
image of $M'$ lies in the boundary of the Gauss image  of $M$ and so has no interior.
Thus $M'$ has zero Gauss curvature.  
By Theorems~\ref{vertical-plane-theorem} and~\ref{convexity-theorem}, $M'$ is a vertical plane or a grim reaper
surface or a tilted grim reaper surface.
\end{proof}

We also use the following result of Spruck and Xiao~\cite{spruck-xiao}*{Theorem~1.5}:

\begin{theorem}\label{spruck-xiao-strip-theorem}
Suppose that $u: \RR \times (-b,b)$ is a complete, strictly convex translator.
Then $u(x,y)\equiv u(x,-y)$.
Also,
\[
    u(x+\zeta,y)- u(\zeta,0)
\]
converges smoothly as $\zeta\to-\infty$ to the tilted grim reaper surface
\begin{equation*}
   \grim_{\theta}: \RR \times (-b,b)\to\RR
\end{equation*}
where $\theta= \arccos(2b/\pi)$ (see~\eqref{tilted-grim-reaper-equation}),
and it converges smoothly as $\zeta\to \infty$ to the tilted grim
reaper surface
\begin{equation*}
   \grim_{-\theta}: \RR \times (-b,b)\to\RR.
\end{equation*}
\end{theorem}

Spruck and Xiao prove that $u(x,y)\equiv u(x,-y)$ by Alexandrov moving planes.
Subsequential convergence of $u(x,y+\zeta)-u(0,\zeta)$ to a tilted grim
reaper surface is relatively easy (see Corollary~\ref{divergent-sequence-corollary}); the difficult part
is showing that such a subsequential limit is a graph over the entire strip $\RR\times (-b,b)$
rather than over a smaller strip.  Spruck and Xiao overcome that difficulty by an ingenious
use of Theorem~\ref{spruck-xiao-gradient-bound} below.

\begin{corollary}\label{spruck-xiao-strip-corollary}
The function $u$ attains its maximum at a point $(x_0,0)$.
\end{corollary}

\begin{proof}
By Theorem~\ref{spruck-xiao-strip-theorem}, $\lim_{|x|\to\infty} u(x,0)=-\infty$, so the function $x\to u(x,0)$
attains its maximum at some $x_0$,  By symmetry, $Du(x_0,0)=0$,
so by convexity, $u$ attains its maximum at $(x_0,0)$.
\end{proof}

The following very useful gradient bound, which plays a crucial role in this paper, 
 is also due to Spruck and Xiao~\cite{spruck-xiao}:

\begin{theorem} \label{spruck-xiao-gradient-bound} Suppose that $M$ is a translating graph in $\RR^{n+1}$, 
and let $v=\left<\ee_{n+1},\nu\right>^{-1}$, where $\nu$ is the upward pointing unit normal to $M$.
Suppose that $\eta:\RR^{n+1}\to\RR$ is an affine function invariant under vertical translations.
Then the function $\eta v$ on $M$ cannot have a local maximum
at an interior point where $\eta$ is positive.
\end{theorem}

Note that if $M$ is the graph of $u$, then $v=\sqrt{1+|Du|^2}$.

\begin{proof}
\newcommand{\drift}{\widetilde{\Delta}}
From the translator equation~\eqref{general-translator-equation}, we see that for $i=1,\dots,n$,
\[
 \Delta x_i = \left< H,\ee_i\right> 
 = \left<(-\ee_{n+1})^\perp, \ee_i\right> 
 = \left< (\ee_{n+1})^T, \ee_i\right> 
 = \left< \ee_{n+1}, \nabla x_i \right>
\]
where $\Delta$ is the Laplacian on $M$ and $\nabla$ is the gradient on $M$.
Thus
\[
  \drift x_i = 0 \quad (i=1,2,\dots,n),
\]
where $\drift$ is the operator on $M$ (sometimes called the Drift Laplacian) 
given by
\[
   \drift f = \Delta f - \left<\ee_{n+1}, \nabla f\right>.
\]
Consequently,
\begin{equation}\label{drift-eta}
  \drift \eta =0.
\end{equation}
According to Martin, Savas-Halilaj and Smoczyk~\cite{mss},
\begin{equation}\label{drift-v}
  \drift v = |A|^2v + 2\frac{|\nabla v|^2}{v}.
\end{equation}
One easily calculates (just as for the Laplacian) that
\begin{equation*}\label{product-rule}
\drift (\eta v) 
=
( \drift \eta) v + \eta (\drift v) + 2\left< \nabla \eta, \nabla v \right>.
\end{equation*}
(This product rule holds for any two smooth functions on $M$.)
Thus by~\eqref{drift-eta} and~\eqref{drift-v},
\begin{equation}\label{drift-eta-v}
\drift (\eta v) 
=
\eta v |A|^2 + 2\eta\frac{|\nabla v|^2}v +  2\left< \nabla \eta, \nabla v \right>.
\end{equation}
At a critical point (of any function), the Laplacian and the Drift Laplacian are equal.
At a critical point of $\eta v$, we have $0=\nabla \eta v=v\nabla \eta +\eta\nabla  v$, 
or $\nabla \eta=-\frac{\eta}{v}\nabla v$. Substitution into the last term in~\eqref{drift-eta-v}
gives
\[
   \Delta (\eta v) = \drift (\eta v) = \eta|A|^2v >0,
\]
so the critical point is not a local maximum.
\end{proof}




\section{Rectangular Boundaries}\label{rectangles-section}

Let 
\[
u=u_{L,b}: [-L,L]\times[-b,b]\to\RR
\]
 be the translator with boundary values $0$.
 
(Existence can be proved in many
 ways.  For example, we can use the continuity method
 starting with $\lambda=0$ to find, for each $\lambda \in [0,1]$, 
 a graph that is minimal with respect to $e^{-\lambda z}\delta_{ij}$.
 The functions are bounded below by $0$, and can be bounded 
 above by a vertical translate of the bowl soliton.  Since $u_{L,b}$ and its vertical translates
 form a $g$-minimal foliation of $[-L,L]\times[-b,b]\times\RR$, we get
 uniqueness by the maximum principle.)

From the translator equation~\eqref{translator-equation}, 
we see that $u_{L,b}$ has no interior local minima, so $u_{L,b}>0$
on the interior of the rectangle.
The function $u_{L,b}$ is smooth on $[-L,L]\times[-b,b]$ except at the corners.
It cannot be $C^2$ at the corners because the translator equation is not satisfied
there since 
\[
 \pdf{u}x=\pdf{u}y= \frac{\partial^2u}{\partial x^2} = \frac{\partial^2u}{\partial y^2} =0
\]
at the corners.   Nevertheless, $u_{L,b}$
is $C^1$ on $[-L,L]\times[-b,b]$ by the following proposition.  (In fact, it is $C^{1,\alpha}$ for
every $\alpha<1$, but we do not need that fact.)

\begin{proposition}\label{rectangles-proposition}
Let $u=u_{L,b}: [-L,L]\times[-b,b]\to\RR$ be the translator with boundary values $0$.
\begin{enumerate}[\upshape (1)]
\item\label{C^1-item} $u$ is $C^1$ everywhere and is smooth except at the four corners.
\item\label{symmetries-item} $u(x,y)\equiv u(-x,y)\equiv u(x,-y)$.
\item\label{x-bigraph-item} $\partial u/\partial x$ and $x$ have opposite signs at interior points where $x\ne 0$.
\item\label{y-bigraph-item} $\partial u/\partial y$ and $y$ have opposite signs at interior points where $y\ne 0$.
\item\label{edge-slope-item} $|Du(L,y)|$ (which is equal to $|Du(-L,y)|$) is a decreasing function of $|y|$.
\end{enumerate}
\end{proposition}

\begin{proof}
Smoothness away from the corners is standard.  To prove that $u$ is $C^1$, let $M$
be the graph of $u$.  Let $p_i\in M$ converge to the corner $q=(-L,-b,0)$.  Translate
$M$ by $(L,b,0)$ and dilate by $|p_i-q|^{-1}$ to get $M_i$.
After passing to a subsequence, the $M_i$ converge smoothly (away from the origin)
to a surface $M'$ such that $M'$ is minimal with respect to the Euclidean metric, 
the boundary $\partial M'$ is the union of the $x$ and $y$ axes, and
 $M'$ lies in the region $\{(x,y,z): x\ge 0, y\ge 0, z\ge 0\}$.   Thus $M'$ is a horizontal quarter-plane.
 This proves~\eqref{C^1-item}.
 
 Properties~\eqref{symmetries-item}, \eqref{x-bigraph-item}, and \eqref{y-bigraph-item} 
 are proved by Alexandrov moving planes.
 
 Since $\pdf{u}y\equiv 0$ on $\{L\}\times (0,b]$ and since $\pdf{u}y<0$ on $(-L,L) \times (0,b]$,
 we see that
 \[
     \frac{\partial^2u}{\partial x \partial y}  \ge 0 \quad\text{on $\{L\} \times [0,b]$}.
\]
Thus
\begin{equation}\label{remixed}
   \frac{\partial^2u}{\partial y \partial x}  \ge 0 \quad\text{on $\{L\} \times [0,b]$.}
\end{equation}
On that edge, $|Du| = - \partial u/ \partial x$.  
 Thus (by~\eqref{remixed}) $|Du(L,y)|$ is a decreasing function of $y\in [0,b]$.
The corresponding statement for $y\in [-b,0]$ follows by symmetry.
This proves~\eqref{edge-slope-item}.
  \end{proof}

\begin{theorem}\label{long-domain-theorem}
Let $\Omega$ be a bounded convex domain in $\RR^p$.   Suppose there
is bounded translator
\[
    u: \Omega\times \RR^q \to \RR
\]
that vanishes on the boundary.   Then $u$ is unique, and therefore $u$
depends only on the first $p$ coordinates.
\end{theorem}

\begin{proof}  
Suppose that $v\ne u$ is another bounded solution.
Choose $(x_i,y_i)\in \Omega\times\RR^q$
such that 
\[
    |u(x_i,y_i)-v(x_i,y_i)| \to \sup |u-v|.
\]
After passing to a subsequence, if we translate the graphs by $(0,-y_i,0)$,
we get convergence (see Theorem~\ref{upper-halfspace-compactness})
 to translaters $\tilde u$ and $\tilde v$ that violate the strong
maximum principle.
\end{proof}

\begin{corollary}\label{height-bound-corollary} 
If
\[
   \liminf_{L\to\infty}u_{L,b}(0,0) < \infty,
\]
then $b<\pi/2$.
\end{corollary}

\begin{proof} Suppose that the liminf is finite.    Then there is a sequence $L(i)\to \infty$ 
such that $u_{L(i),b}$ converges smoothly to a translator
\[
   u: [-b,b]\times\RR\to\RR
\]
with $u=0$ on the boundary of the strip and $u(0,0)=\max u < \infty$.

By Theorem~\ref{long-domain-theorem}, $u(x,y)$ is independent of $x$ and thus 
its graph is a portion of the grim reaper surface. 
Consequently $b<\pi/2$.
\end{proof}

In the next section, we produce a $\Delta$-wing on $\RR \times (-b,b)$ by taking a subsequential
limit of $u_{L,b}-u_{L,b}(0,0)$ as $L\to\infty$. 
Standard curvature bounds give smooth subsequential convergence, but
conceivably the domain of the limit function might be a thinner strip inside $\RR \times (-b,b)$.
The following gradient bound guarantees that such thinning does not happen.

 \begin{proposition}\label{rectangle-gradient-bound}
 Let $B\ge \pi/2$.  There is a constant $C(B)<\infty$ such that
 \[
     (b- |y|) \sqrt{ 1 + |Du_{L,b}(x,y)|^2} \le C(B)
\]
for all $b \le B$, $L\ge 1$, and $(x,y)\in [-L,L] \times [-b,b]$.
\end{proposition}

\begin{proof}
By the symmetry, it suffices to prove the bound for $y\ge 0$.
Let 
\begin{align*}
&w=w_{L,b}: [-L,L] \times [0,b] \to \RR,  \\
&w(x,y) = (b-y)\sqrt{1 + |Du_{b,L}(x,y)|^2}.
\end{align*}
By Theorem~\ref{spruck-xiao-gradient-bound}, the maximum of $w$ occurs at a point $(\tilde x, \tilde y)$ on one of the four edges of the rectangle $[-L,L] \times [0,b]$.
It cannot occur on the edge $[-L,L] \times \{b\}$ since the function vanishes there.

By Proposition~\ref{rectangles-proposition}\eqref{edge-slope-item}, 
it cannot occur on the edges $\{L\}\times(0,a]$ or $\{-L\}\times(0,a]$.

Thus it occurs at a point in $ [-L,L] \times \{0\}$.

Now suppose that Proposition~\ref{rectangle-gradient-bound} is false.  Then there exist sequences $L(i)\to\infty$
and $b(i)\le A$ such that 
\[
  \max_{ [-L(i),L(i)]\times[-b(i),b(i)]}w_{L(i), b(i)} \to \infty.
\]
By the forgoing discussion, the maximum is attained at a point $(x_i,0)$.
Let $M_i$ be the graph of $u_{L(i),b(i)}$ and let $p_i=(x_i,0,u_{L(i),b(i)}(x_i,0))$.

Now
\[
  |Du_{L(i),b(i)}(x_i,0)| \to\infty,
\]
and $\pdf{}y u_{L(i),b(i)}(x_i,0)\equiv 0$ (by the symmetry), so
\[
  \text{$\Tan(M_i,p_i)$ converges to the plane $x=0$}.
\]

Translate $M_i$ by $-p_i$ to get a translator $M_i'$.
By passing to a subsequence, we can assume that the $M_i'$ converge to a limit translator $M'$.
We can also assume that $\dist(0,\partial M_i')$ converges to a limit $\delta \in [0,\infty]$.

By the maximum principle (if $\delta>0$) or the boundary maximum principle (if $\delta=0$),
\[
   \Tan(M',p) = \Tan(M',0) = \{x=0\}
\]
for $p$ in a small neighborhood of $0$.  Thus there is a neighborhood $G$
of $0$ such that 
\[
  M'\cap G \subset \{x=0\}.
\]
By unique continuation, $M'$ contains the entire plane $\{x=0\}$, which is impossible
since $M'$ is contained in the slab $\RR \times [-b,b] \times \RR$.
\end{proof}

\section{Existence of $\Delta$-Wings}\label{existence-section}


\begin{theorem} \label{existence-theorem}
For every $b\in \left(\frac \pi2, +\infty \right)$ there exists a complete translator 
\[
   u^b: \RR \times (-b,b)\to\RR
\]
with  the following properties:
\begin{enumerate}[\upshape (a)]
\item $u^b(x,y)\equiv u^b(-x,y)\equiv u^b(x,-y)$.
\item The Gauss curvature of the graph is everywhere positive.
\item\label{gradient-bound-item}  For $B\ge b$, 
\[
   \sup_{(x,y)} (b-|y|)\sqrt{1+ |Du^b(x,y)|^2} \le C(B)
\]
where $C(B)<\infty$ is as in Proposition~\ref{rectangle-gradient-bound}.
\end{enumerate}
\end{theorem}

\begin{proof}
As in Section~\ref{rectangles-section}, let
$
   u = u_{L,b} :    [ -L, L] \times [-b,b] \longrightarrow \RR
$
be the solution to the translator equation with
boundary values $0$.

By Proposition~\ref{rectangle-gradient-bound}, $|Du_{L,b}|$ is bounded (independent of $L$) on compact subsets
of $\RR \times (-b,b)$.  Thus every sequence of $L$'s tending to infinity has a subsequence $L(i)$
such that $u_{L(i),b}-u_{L(i),b}(0,0)$ converges smoothly to a limit translator
\[
   u^b: \RR \times (-b,b) \to \RR.
\]
(Later we will show that the limit $u^b$ does not depend on the choice of the sequence $L(i)$;
see Proposition~\ref{symmetric-uniqueness}.)

Since $u_{L(i),b}(0,0)$ tends to infinity (see Corollary~\ref{height-bound-corollary}), 
the graph of $u^b$ is a complete translator.
The symmetries
\begin{equation}\label{u^a-symmetries}
  u^b(x,y)\equiv u^b(-x,y)\equiv u^b(x,-y)
\end{equation}
of $u^b$ follow from the corresponding symmetries of the $u_{L,b}$.

Since $b>\pi/2$, we see that $u^b$ is not a grim reaper surface.
The symmetries~\eqref{u^a-symmetries} imply that $u^a$ is not a tilted grim reaper surface.
Hence the Gauss curvature of the translator $u^a$ is everywhere positive
by Theorem~\ref{convexity-theorem}.
\end{proof}

\section{Uniqueness of $\Delta$-Wings}\label{symmetric-delta-wings-section}

\begin{proposition}\label{different-intervals-proposition}
Suppose that
\begin{align*}
 &u: \RR \times (a,b)\to \RR, \\
 &\tilde u:  \RR \times (\tilde a, \tilde b) \to \RR
\end{align*}
are complete, strictly convex translators. 
\begin{enumerate}[\upshape(1)]
\item\label{shifted-item} If $a<\tilde a <b< \tilde b$, then $u-\tilde u$ has no critical points.
\item\label{nested-item} If $a<\tilde a<\tilde b<b $, then $u-\tilde u$ has at most one critical point.
If there is a critical point, then $D^2u\ne D^2\tilde u$ at that point.
\end{enumerate}
\end{proposition}

\begin{proof}
This follows from the tilted-grim-reaper-like behavior of $u$ and $\tilde u$ as $|x|\to\infty$ (Theorem~\ref{spruck-xiao-strip-theorem}). That together with a classical theorem of Rado (Theorem~\ref{rado-theorem} below) gives~\eqref{shifted-item}. To prove~\eqref{nested-item}, consider the function \[ F:= u - \tilde u: [-L,L] \times (\tilde a, \tilde b) \to \RR. \] By the asymptotic behavior of $u(x,y)$ and $\tilde u(x,y)$ as $|x| \to \infty$, we can choose $L$ large enough so that the restriction of $F$ to $\{L\}\times (\tilde\alpha,\tilde\beta)$ has at most one local minimum, and likewise for $\{-L\}\times (\tilde\alpha,\tilde\beta)$. (Note that $F=+\infty$ on the other two edges of the rectangular domain.) We can also choose $L$ so that there are no critical points of $F$ on $\{x=\pm L\}$. It follows by Morse Theory (see~\cite{scherkon}*{Theorem~7.1}(17)) that $F$ has at most one interior critical point, and it is a simple critical point. (Note that in that theorem, $c_0(M_a)$ is at most $2$, $c_1(M_a)$ is nonnegative, and $\chi(M_a)$ is at least one.)
\end{proof}

\begin{corollary}\label{same-x-corollary}
If $u, \hat{u}: \RR \times (-b,b) \to\RR$ are complete translators and if
\[
    Du(x_0,y_0)=D\hat{u}(\hat{x}_0,\hat{y}_0),
\]
then $y_0=\hat{y}_0$.
\end{corollary}

\begin{proof} If $y_0 \ne \hat{y}_0$, then the functions 
 $u$ and  $\tilde u(x,y):=\hat{u}(x,y+ (\hat{y}_0-y_0))$
 would violate Proposition~\ref{different-intervals-proposition}\eqref{shifted-item}.
\end{proof}

\begin{theorem}[Rado's Theorem]\label{rado-theorem}
Suppose that $U$ is a simply connected open subset of $\RR^2$,
and that
\[
   f,g : U\to \RR
\]
are smooth functions such that 
the graphs of $f$ and $g$ and their vertical translates are minimal with
respect to a smooth Riemannian metric.
\begin{enumerate}[\upshape (1)]
\item If $\{f-g=c\}\cap \partial U$ has fewer than $4$ ends,  then $U\cap\{f-g=c\}$ 
has no point at which $Df=Dg$.
\item If $\{f-g=c\}\cap\partial U$ has fewer than $6$ ends for each $c$,
then $f-g$ is a Morse function on $U$.
\end{enumerate}
\end{theorem}

See~\cite{rado-book}.

\begin{theorem}[Uniqueness of Symmetric $\Delta$-Wings]\label{symmetric-uniqueness}
For each $b>\pi/2$, there exists a unique translator 
\[
  u: \RR \times (-b,b)\to\RR
\]
such that $u(x,y)\equiv u(-x,y)$.
\end{theorem}

\begin{proof}
We proved existence in Theorem~\ref{existence-theorem}, 
so it suffices to prove uniqueness.  Suppose that $u$ and $\hat{u}$ both satisfy the
hypotheses.
Let $0<y_0<b$.  Then 
\[
  \pdf{u}y(0,y_0)<0.
\]
Since $\pdf{u}y(0,y_0)$ decreases from $0$ to $-\infty$ on the
interval $[0,b)$, there is an  $\hat{y}_0\in [0,b)$ such that 
\[
    \pdf{\hat{u}}{y}(0,\hat{y}_0)= \pdf{u}y(0,y_0).
\] 
By the symmetry assumption, 
\[
  \pdf{u}x = \pdf{\hat{u}}{x} =0 \quad \text{on $\{0\} \times (-b,b)$}.
\]
Thus $Du(0,y_0)=D\hat{u}(0,\hat{y}_0)$.
Hence by Corollary~\ref{same-x-corollary}, $\hat{y}_0=y_0$, so
\[
   Du(0,y_0)=D\hat{u}(0,y_0).
\]

Since $y_0\in [0,b)$ was arbitrary, we have proved
 that $Du(0,y)\equiv D\hat{u}(0,y)$ for all $y\in [0,b)$.  Hence (since $u(0,0)=\hat{u}(0,0)$)
we have $u\equiv \hat{u}$ by Cauchy-Kowalevski.
\end{proof}

\begin{corollary}[Continuous Dependence]\label{continuous-dependence-corollary}
For $b>\pi/2$, let $u^b:\RR \times (-b,b) \rightarrow \R$ be the unique translator such that
$u^b(0,0)=0$ and such that $u^b(x,y)\equiv u^b(-x,y)$.
Then $u^b$ depends continuously on $b$: as $b\to d$, the function $u^b$ converges
smoothly to $u^d$.
\end{corollary}

\begin{proof}
Let $b(i)\to d$, where $b(i)$ and $d$ are in the interval $(\pi/2,\infty)$. 
By Theorem~\ref{existence-theorem}\eqref{gradient-bound-item}, $u^{b(i)}$ converges smoothly (after passing to a subsequence)
to a complete translator
\[
    u: \RR \times (-d,d)\to \RR.
\]
By~Theorem~\ref{symmetric-uniqueness}, $u=u^d$. Since the limit is independent of the sequence and subsequence, we are done.
\end{proof}

\begin{lemma}\label{spine-comparison-lemma}
Let $a<b$ and let
\begin{align*}
&u: \RR \times (-a,a)\to\RR, \\
&v: \RR\times (-b,b) \to\RR
\end{align*}
be complete translators with $u(0,0)=v(0,0)$ and $Du(0,0)=Dv(0,0)=0$.
Then
\[
   u(x,0) > v(x,0)
\]
for all $x\ne 0$.
\end{lemma}

\begin{proof}
Note that $\pdf{u}y(x,0)=0=\pdf{v}y(x,0)$ by $(x,y)\to (x,-y)$ symmetry.
Thus if 
\[
   \pdf{u}x(x,0)=\pdf{v}x(x,0),
\]
then $(x,0)$ is a critical point of $u-v$.  
By Proposition~\ref{different-intervals-proposition}\eqref{nested-item}, 
 $u-v$ has only one critical point (namely $(0,0)$).
Thus
\[
   \text{$\pdf{u}x(x,0)\ne \pdf{v}x(x,0)$ for all $x\ne 0$.}
\]
By Theorem~\ref{spruck-xiao-strip-theorem},
\[
  \lim_{x\to\infty} \left( \pdf{u}x(x,0)-\pdf{v}x(x,0) \right) > 0.
\]
Thus
\[
    \text{$\pdf{u}x(x,0) > \pdf{v}x(x,0)$ for all $x> 0$}.
\]
Integrating from $0$ gives 
\[
    \text{$u(x,0) > v(x,0)$ for all $x> 0$}.
\]
Exactly the same argument shows that $u(x,0) > v(x,0)$ for all $x<0$.
\end{proof}

\begin{theorem}[Existence and Uniqueness of $\Delta$-Wings]\label{unique-Delta-wings-theorem}
Let $b>\pi/2$. Then, modulo translations, there is a unique complete translator
 $u:\RR \times (-b,b)\rightarrow \RR$ that is not a tilted grim reaper.
 \end{theorem}
 
\begin{proof}
We already proved existence (Theorem~\ref{symmetric-uniqueness}), 
so it suffices to prove that if $u:\RR\times(-b,b) \to \RR$ is any
complete translator (other than a tilted grim reaper), 
then (after a translation) $u$ is equal to the unique symmetric translator $u^b$
given by Theorem~\ref{symmetric-uniqueness}.
By Corollary~\ref{spruck-xiao-strip-corollary}, $u$ attains its maximum.  By translating, we can assume that
\[
    \max u = u(0,0)=0.
\]
By Lemma~\ref{spine-comparison-lemma},
\[
   u^c(x,0) \le u(x,0) \le u^a(x,0) \quad\text{for all $a<b<c$ and all $x\in \RR$.}.
\]
Letting $a$ and $c$ tend to $b$ gives (see Corollary~\ref{continuous-dependence-corollary})
\[
   u(x,0) = u^b(x,0) \quad\text{for all $x$}.
\]
Since $\pdf{u}y(x,0)=0=u^b_y(x,0)$ for all $x$, we see (by Cauchy-Kowalevski, for example) that $u=u^b$.
\end{proof}

\section{Non-existence of translating graphs over half-planes}

In this section, we prove that no complete translators in $\R^3$ are
graphs over a half-plane. The first result that we need in the proof is about
the image under the Gauss map of such a translator. Let $M$ be a complete translator and
$\nu: M \rightarrow \SS^2$ its Gauss map. If we assume that $M$ is a graph, then
clearly $\nu(M) \subseteq H^+$, where $H^+$ represent the upper hemisphere.

The next lemma  says that the Gauss image of a complete translating graph is a domain 
in $H^+$ bounded by 0, 1, or 2 great semicircles.

\begin{lemma}[Gauss map lemma] \label{Gauss-map-lemma}
Let $M\subset \RR^3$ be a complete, strictly convex translator that is 
the graph of a function $u:\Omega\subset\RR^2\to \RR$.
Then the Gauss map is a diffeomorphism from $M$  onto an open subset $W$ of the
upper hemisphere $H^+$.
Furthermore, one of the following holds:
\begin{enumerate}[\upshape (1)]
\item $\nu(M)$ is the entire upper hemisphere.
\item $\nu(M)$ is one of the two components of $H^+\setminus C$, where $C$ is a great semicircle in $H^+$.
\item $\nu(M)$ is the region between $C'$ and $C''$, where $C'$ and $C''$ are two disjoint great semicircles in $H^+$.
\end{enumerate}
\end{lemma}

\begin{proof}
That $\nu$ maps $M$ diffeomorphically onto its image holds for any complete, strictly convex $M$.
The following two statements are immediate consequences:
\begin{enumerate}[\upshape (i)]
\item If $p_i\in M$ and if $\nu(p_i)$ converges to a limit $v\notin \nu(M)$, then $|p_i|\to \infty$.
\item  If $p_i\in M$ and if $|p_i|\to\infty$, then all the subsequential limits of $\nu(p_i)$ lie in $\partial W$.
\end{enumerate}

Now suppose that $v\in H^+\cap \partial W$.
Choose $p_i\in M$ so that $\nu(p_i)\to v$.  By (i), $|p_i|\to\infty$.
By passing to a subsequence, we can assume that $M-p_i$ converges smoothly
to a translator $M'$.  Let $\nu'$ be the Gauss map of $M'$.  
  By (ii), $\nu'(M')$ is contained in $\partial W$.
Thus $\nu(M')$ has no interior, so $M'$ has Gauss curvature $0$ at all points.
Thus $M'$ is a tilted grim reaper or a vertical plane.  Since $\nu'(0)=v\in H^+$,
$M'$ must be a tilted grim reaper.  Thus $\nu'(M')$ is a great semicircle $C$ containing $v$.

We have shown that every point in $H^+\cap \partial W$ 
lies in a great semicircle contained in $H^+\cap \partial W$.
Thus $H^+\cap \partial W$ is the union of a collection $\mathcal{C}$ of great semicircles.

If $\mathcal{C}$ is empty, (1) holds.  If $\mathcal{C}$ has just one semicircle, then (2) holds.

Suppose that $\mathcal{C}$ contains more than one semicircle.  Let $C'$ and $C''$ be two of them.
Since $W$ is connected, it lies in one of the components of $H^+\setminus (C'\cup C'')$,
and its closure contains $C'\cup C''$.
It follows immediately that $C'$ and $C''$ do not intersect. 
We have shown that all the semicircles in $\mathcal{C}$ are disjoint.  Since $W$ is connected,
there cannot be more than two semicircles in $\mathcal{C}$.
\end{proof}

\begin{corollary}
Suppose that $p_i\in M$ and that $\nu(p_i)$ converges to a point in $H^+\cap\partial W$.
Let $C$ be the great semicircle in $H^+\cap \partial W$ that contains $v$.
Then $M-p_i$ converges to the unique tilted grim reaper $M'$ such that
$\nu'(M')=C$ and $\nu'(0)=v$ (where $\nu'$ is the Gauss map of $M'$.)
\end{corollary}
In particular, we get convergence without having to pass to a subsequence.

The next lemma is inspired by results of Spruck and Xiao in~\cite{spruck-xiao}.

\begin{lemma} \label{spruck-xiao-type-lemma}
Let $\Omega$ be a convex open subset of $\R^2$ containing $[0,\infty)\times [-a,a]$.
Suppose that $u:\Omega \rightarrow \R$ is a translator and suppose
that $[0,\infty)\times [-a,a]$ contains a closed, connected, unbounded subset
$C$ such that $|Du|$ is bounded on $C$. Then $|Du|$ is bounded on
$[0,\infty)\times [-a,a].$
\end{lemma}

\begin{proof}
We can assume that $C$ contains the origin; otherwise replace $C$ 
by the union of $C$ and a segment joining the origin with a point of $C$.
Moreover, we can assume that $C$ has no interior; otherwise replace
$C$ by its topological boundary in $\R^2$.

Take $b>a$ so that $J_b:=\{0\} \times [-b,b] \subset \Omega. $ Let
\begin{equation}\label{lambda}
    \lambda= \sup_{J_b\cup C} |Du| <\infty.
\end{equation}
For $\epsilon >0$, let $L(\epsilon)$ the line $$y=-b+\epsilon \, x.$$
Let $W(\epsilon)$ be the union of the bounded components of $\R^2\setminus
(C \cup J_b \cup L(\epsilon))$ that lie above $L(\epsilon)$. Note that
$W(\epsilon)$ lies in the triangle $T(\epsilon)$ determined by the lines
$x=0$, $y=a$ and $L(\epsilon)$.

Thus the function $\psi_\epsilon: \overline{W(\epsilon)} \rightarrow \R$
$$\psi_\epsilon(p)=\sqrt{1+|Du(p)|^2} \cdot \dist(p,L(\epsilon))$$
attains its maximum at a point $p_\epsilon$.  The point $p_\epsilon$ need not  be
unique, nor need it depend continuously on $\epsilon$. The point 
$p_\epsilon$ does not lie on $L(\epsilon)$ because $\psi_\epsilon$
vanishes on $L(\epsilon)$.
By Theorem~\ref{spruck-xiao-gradient-bound}, it cannot be in the interior 
of $W(\epsilon)$. Thus $p_\epsilon \in C \cup J_b$ (and in the triangle 
$\overline{T(\epsilon)}$.) Then, for every point $p \in \overline{W(\epsilon)}$, we
have
\begin{multline}
\psi_\epsilon(p) \leq \psi_\epsilon(p_\epsilon) \leq \lambda \, \dist (p_\epsilon, L(\epsilon)) \leq 
\lambda \, \dist (p_\epsilon, L(0)) \leq 2 \, \lambda \, b,
\end{multline}
where $L(0)$ is the horizontal line $y=-b$ and where $\lambda$ is given by~\eqref{lambda}.

Now let $W$ be the union of the connected components of 
$$ ([0,\infty) \times [-b, \infty)) \setminus (C\cup J_b) $$
that lie in below the line $y=b$.

Suppose that $(x,y) \in W$. Then $(x,y) \in \overline{W(\epsilon)}$ for all sufficiently small
$\epsilon>0$, so 
$$\psi_\epsilon(x,y) \leq 2 \, \lambda \, b, \quad \mbox{for all such  $\epsilon.$}$$ 
Letting $\epsilon \to 0$ and using that $\dist((x,y),L(0))=y+b \geq b-|y|,$ then we have 
\begin{equation} \label{eq:aqui}
\sqrt{1+|Du(x,y)|^2} \cdot (b-|y|) \leq 2 \, \lambda \, b, \quad \forall (x,y) \in W.
\end{equation}
By continuity, the inequality also holds for all $(x,y) \in \overline W.$ Applying the same argument 
to the lines $y=b-\epsilon x$ shows that the inequality~\eqref{eq:aqui} also holds for all
$(x,y) \in \overline{W^*}$, where $W^*$ is the union of the components of 
$$ ([0,\infty) \times (-\infty,b]) \setminus (C\cup J_b) $$
that lie above the line $y=-b$. But it is not hard to see 
\[
    \overline W \cup \overline{W^*} =[0,\infty) \times[-b,b],
\]
so inequality \eqref{eq:aqui}
 holds for every $(x,y) \in [0,\infty) \times[-b,b].$
\end{proof}

\begin{corollary}
\label{spruck-xiao-gradient-bound2} Suppose,  in addition to the hypotheses of Lemma~\ref{spruck-xiao-type-lemma},
 that the graph $M$ of $u$ is complete and that the domain $\Omega$ is a plane or a half-plane. 
 Let $\{ x_n\}$ be a sequence of real numbers such that $x_n \to \infty.$ Then a subsequence of 
$M-(x_n,0,u(x_n,0))$ converges smoothly to a complete translator $M'$ that is the graph of a function
$u': \Omega' \rightarrow \R$, where $\Omega'$ is the limit of the domains $\Omega-(x_n,0).$
\end{corollary}

\begin{proof}
Suppose first that $\Omega=\{(x,y) \in \R^2 \, : \, y>c\}$, for some constant $c \in \R.$

For all $\alpha >c$ sufficiently close to $c$ and all $\beta$ sufficiently large,
the set $C$ in Lemma \ref{spruck-xiao-type-lemma} will be contained in $[0,\infty) \times [\alpha, \beta].$

Consequently $|Du|$ is bounded above on $[0,\infty) \times [\alpha, \beta].$ So, any
subsequential limit $M'$ of the indicated kind is a graph whose domain includes $[0,\infty) \times [\alpha, \beta].$
Since $\alpha$ and $\beta$ are arbitrary, then $M'$ is a graph over $\Omega.$

The other cases are similar, but easier.
\end{proof}

\begin{remark}
\label{re:good2} The proof of the previous corollary gives a bit more information
about the limit function $u'$:
\begin{enumerate}[(1)]
\item If $\Omega$ is a halfplane of the form $\{(x,y) \in \R^2 \, : \, y>c\}$, then
 $|Du'|$ is bounded on strips of the form 
$\{(x,y) \in \R^2 \, : \, \alpha \leq y \leq \beta \}$, provided that $\alpha >c.$
\item If $\Omega$ is the plane, or if $\partial \Omega$ consists of a non-horizontal line,
then $\Omega'=\R^2$ and $|Du'|$ is bounded on all horizontal strips $\{(x,y) \in \R^2 \, : \, \alpha \leq y \leq \beta \}.$
\end{enumerate}
\end{remark}

\begin{proposition}\label{Gauss-image-is-hemisphere} 
Suppose that $M$ is a complete, strictly convex translator and that $M$ is a graph of 
$u: \Omega \rightarrow \R$, where $\Omega$ is either all of $\R^2$ or a halfplane. Then the Gauss
image $\nu(M)$ is the entire upper hemisphere $H^+.$
\end{proposition}

\begin{proof}
Suppose not. Then by Lemma \ref{Gauss-map-lemma}, the boundary of $\nu(M)$ contains
a great semicircle $C$ in the upper unit hemisphere. By rotating, we can
assume that the endpoints of the semicircle are $(0,1,0)$ and $(0,-1,0).$
 Let $\Gamma$ be the set of points in $\Omega$ such that $\pdf{u}y = 0$, i.e., the inverse image under $\nu$ of
$\{v \in H^+ \, : \, \langle v , (0,1,0) \rangle = 0\}.$
By Lemma \ref{Gauss-map-lemma}, $\Gamma$ is a smooth, properly embedded curve in $\Omega.$ By
translation, we may assume that $(0,0) \in \Gamma$. Note the strict convexity implies that
$$y \mapsto \pdf{u}y(x,y)$$ is strictly decreasing (for each $x$).
 Thus if $\Gamma$ intersects the line $\{(x, y) \,: \,y \in  \R\}$, it
intersects it in a single point $(x,y(x)).$ That is, we can parametrize $\Gamma$ as
$$\{(x,y(x)) \,: \, x  \in I \}$$
where $I$ is an open interval (possibly infinite) containing 0.
 
By Lemma \ref{Gauss-map-lemma} again, as $x$ tends to one of the endpoints of $I,$ say the right
endpoint, the surfaces $M -(x,y(x),u(x,y(x))$ converge smoothly to a grim reaper
surface $M'$ through $(0,0,0)$ such that $\nu'(M') = C.$ Thus $M'$ is the graph of a function
\[
     u' : \R \times (-a, a) \longrightarrow  \R.
\]
The strip is horizontal because the endpoints of $\nu'(M')$ are $(0,1,0)$ and $(0,-1,0).$
The strip is symmetric about the line $X = \R \times \{0\}$ because
$$u'_y(0,0) = \lim \pdf{u}y(x,y(x)) = 0.$$
It follows that the curves $\Gamma-(x,y(x))$ converge smoothly to $X$ as $x$ tends to the
right endpoint of $I.$ Thus the right endpoint of $I$ is $+\infty.$
 
For $t > 0$, let $L(t)$ be the line through $(0,0)$ and $(t,y(t)),$ let
\[
     \Gamma[0,t] = \{(x, y) \in \Gamma \, : \, 0 \leq x \leq t\} = \{(x, y(x))  \, : \, 0 \leq x \leq t\} 
\]     
and let $K(t)$ be the closed bounded subset of $[0,t]\times \R$ determined by $\Gamma \cup L(t).$
(Thus for each vertical line $V$ in $[0,t]\times \R$, $V \cap K(t)$ is the closed segment whose
endpoints are $V \cap \Gamma$ and $V \cap L(t).$)

Note that as $t \to \infty$ the line $L(t)$ converges to the horizontal line $X $,
and thus $K(t)$ converges to the set
$$K := \{(x,y) \, : \, x \geq 0, \mbox{ and $0 \leq y \leq y(x)$ or $y(x) \leq y \leq 0$} \}.$$
Let $p(t)$ be a point where the function 
\begin{align*}
&\varphi^t: K(t) \rightarrow \R, \\
&\varphi^t(p)=\sqrt{1 + |Du(p)|^2} \, \dist(p,L(t))
\end{align*}
attains its maximum. (Of course $p(t)$ need not depend continuously on $t.$)

Now $p(t)$ cannot be on $L(t)$ since the function $\varphi^t$ vanishes on $L(t).$ By 
Corollary~\ref{spruck-xiao-gradient-bound}, $p(t)$ cannot be in the interior of $K(t).$ Thus
\begin{equation}
p(t)\in  \Gamma[0,t].
\end{equation}
We claim that 
\begin{equation}
\label{eq:alpha}
\alpha:= \sup_{t>0}  \sqrt{1+|Du(p(t))|^2} \dist (p(t) ,L(t)) <\infty
\end{equation}
For suppose to the contrary that we can find a sequence $t_n \nearrow \infty$
such that
\begin{equation}
\label{eq:alphano}
\sqrt{1+|Du(p_n)|^2} \dist (p_n ,L_n) \to  \infty
\end{equation}
where $p_n=p(t_n)$ and $L_n=L(t_n)$.
Since $|Du(p_n)| \to |Du'(0)| <\infty,$ we see from \eqref{eq:alphano} that
\begin{equation}\label{eq:alphano1}
\dist(p_n,L_n) \to \infty.
\end{equation}
Since $\Gamma-p_n$ converges smoothly to the line $X$, \eqref{eq:alphano1}
implies (after passing to a further subsequence) that $K(t_n)-p_n$ 
converges to a halfplane $Q$ bounded by $X$.

Let $q$ be a point in the interior of $Q$. Then, for all sufficiently large
$n$, one has $p_n +q \in K(t_n)$. Thus
\begin{align*}
\sqrt{1+|Du(p_n)|^2} \, \dist(p_n,L_n) 
&\geq \sqrt{1+|Du(p_n+q)|^2} \, \dist(p_n+q,L_n)  \\
&\geq \sqrt{1+|Du(p_n+q)|^2} \, \left(\dist(p_n,L_n)-|q|\right).
\end{align*}
Dividing by $\dist(p_n,L_n)$, which tends to $\infty$ by \eqref{eq:alphano}, 
and letting $n \to \infty$ gives
$$ \sqrt{1+|Du'(0)|^2} \geq \sqrt{1+|Du'(q)|^2}, \quad \forall q \in Q,$$
which is absurd because the graph of $u'$ is a tilted grim reaper. This
contradiction proves \eqref{eq:alpha}.

Now consider a point $p=(x,y(x))$ in $\Gamma \cap\{x \geq 0\}.$ 
Then $p \in K(t)$ for $t \geq x$, so for all $t \geq x$, 
$$ \sqrt{1+|Du(x,y(x))|^2} \dist((x,y(x)),L(t)) \leq \alpha.$$
Letting $t \to \infty$ gives 
$$ |y(x)| \leq \sqrt{1+|Du(x,y(x))|^2} \, |y(x)| \leq \alpha,$$
and therefore $$\sup_{x \geq 0} |y(x)| \leq \alpha.$$
Now let $\widehat M$ the subsequential limit of $M-(x_n,0,u(x_n,0))$
as $x_n \to \infty$. By Corollary \ref{spruck-xiao-gradient-bound2} and Remark \ref{re:good2},
$\widehat M$ is a complete graph defined over a halfplane or over
all of $\R^2$. But by Lemma \ref{Gauss-map-lemma}, since $M$ is convex, 
$\widetilde M$ is either a vertical plane or a tilted grim reaper, a contradiction.
\end{proof}

\begin{theorem} \label{halfplane-theorem}
No complete translator is the graph of a function over a halfplane.
\end{theorem}

\begin{proof}
We prove it by contradiction using Alexandrov moving planes.
Suppose there is a complete translator $M$ that is the graph of a function
\[
  u: \{(x,y): y>0\}\to\RR.
\]
By Proposition~\ref{Gauss-image-is-hemisphere}, the Gauss map image $\nu(M)$ is the entire
upper hemisphere.  Thus the only limits of translates of $M$ are vertical planes.

If $p_n=(x_n,y_n,z_n) \in M$ is a divergent sequence with 
\begin{equation}\label{eq:*}
y_n \leq c <\infty,
\end{equation}
then the sequence $M-p_n$ converges (subsequentially) to a vertical plane passing through the origin.
Since the plane is contained in $\{ y \geq -c\}$, it must be the 
the plane $\Pi_0=\{y=0\}.$   
Thus
\begin{equation}\label{oops}
\begin{gathered}
\text{If $(x_n,y_n,z_n)\in M$ diverges and if $y_n$ is bounded,} \\
\text{then $\pdf{u}y(x_n,y_n)\to\infty$.}
\end{gathered}
\end{equation}

In particular, there is an $\eta>0$ such that 
\[
    \text{$\pdf{u}y(x,y)>0$ for $0<\eta<y$}.
\]

Let $\Ss$ be the set of $(x,y,y')$ such that 
\begin{gather*}
  \text{$0<y=y'$ and $\pdf{u}y(x,y)=0$, or} \\
  \text{$0<y<y'$ and $u(x,y)\ge u(x,y')$}.
\end{gather*}

Since $\nu(M)$ is the upper hemisphere, there is a point $(x,y)$ 
with $Du(x,y)=0$.   Thus $(x,y,y)\in \Ss$, so $\Ss$ is nonempty.

Let 
\[
    s = \inf \{ (y+y')/2:  (x,y,y')\in \Ss\}.
\]
We claim that the infimum is attained.  To see this, let $(x_i,y_i,y_i')$ be a sequence in $\Ss$ such that
\[
  (y_i+y_i')/2 \to s.
\]
Note by~\eqref{oops} that $x_i$ is bounded.
It follows (also by~\eqref{oops}) that $y_i$ is bounded away from $0$.  
Thus (after passing to a subsequence) $(x_i,y_i, y_i')$ converges to a limit $(\hat{x},\hat{y},\hat{y}')$ in $\Ss$
with $s=(\hat{y}+\hat{y}')/2$.

Now 
\[
   \text{$u(x,y) \le u(x,2s-y)$ for all $y\in (0,s]$}
\]
with equality at $(\hat{x},\hat{y})$.  By the strong maximum principle (if $\hat{y}<\hat{y}'$) or the
strong boundary maximum principle (if $\hat{y}=\hat{y}'$), 
\[
   \text{$u(x,y) = u(x,2s-y)$ for all $x\in \RR$ and $y\in (0,s]$},
\]
which is clearly impossible.
\end{proof}

\section{The Classification Theorem}

\begin{theorem}
For every $b>\pi/2$, there is (up to translation) a unique complete, strictly convex
translator 
\[
   u^b: \RR \times (-b,b)\to\RR.
\]
Up to isometries of $\RR^2$, the only other complete translating graphs
are the grim reaper surface, the tilted grim reaper surfaces, and the
rotationally symmetric graphical translator (i.e., the bowl soliton).
\end{theorem}

\begin{proof}
Let $u:\Omega\to \RR$ be a complete translator that is not
a grim reaper surface or tilted grim reaper surface.
By Theorem~\ref{convexity-theorem},  its graph is strictly convex.

By~\cite{shari}, $\Omega$ is a strip, a halfplane, or all of $\RR^2$.
Theorem~\ref{unique-Delta-wings-theorem} gives existence and uniqueness (up to rigid motion) of 
complete, strictly convex $u^b:\RR\times (-b,b)\to\RR$.

By Theorem~\ref{halfplane-theorem}, $\Omega$ cannot be a halfplane.

It remains only to consider the case when $\Omega=\RR^2$.
X. J. Wang~\cite{wang} showed that (up to translation) the only
entire, convex translator is the bowl soliton.
\end{proof}


\newcommand{\rr}{\mathbf{r}}
\newcommand{\kk}{\mathbf{k}}

\newcommand{\uu}{\mathbf{u}}
\renewcommand{\vv}{\mathbf{v}}

\section{Higher Dimensional $\Delta$-Wings with Prescribed Principal 
Curvatures at the Apex}\label{higher-delta-wings}

In this section, we prove the following theorem:

\begin{theorem}\label{high-dimensional-delta-wings}
Let $k_1,\dots,k_n$ be nonnegative numbers whose sum is $1$.   Then there
is an open subset $\Omega$ of $\RR^n$ and a complete, properly embedded 
translator $u: \Omega\to \RR$
with the following
properties:
\begin{enumerate}[\upshape (1)]
\item\label{max-point-item} $\max u = u(0)=0$.
\item\label{principal-curvatures-item} $D^2u(0)$ is a diagonal matrix whose $ii$ entry is $(-k_i)$ for each $i$.
\item\label{even-function-item} $u$ is an even function of each of its coordinates.
\item\label{alexandrov-item} If $k_i=0$, then $u$ is translation-invariant in the $\ee_i$ direction. 
   If $k_i>0$, then $D_iu(x)$ and $x_i$ have opposite signs wherever $x_i\ne 0$.
\item\label{axis-rotation-item} 
If $k_i=k_j$, then $u$ is rotationally invariant about the plane $\{x_i=x_j=0\}$.
\item\label{domain-item} The domain of $u$ is 
either all of $\RR^n$ or a slab of the form $\{x: -b<x_i<b\}$ for some $i$.
  In the latter case, $k_i>k_j$ for all $j\ne i$.
\end{enumerate}
\end{theorem}

\begin{corollary}
If $k_1= k_2 \ge k_3 \ge \dots k_n$, then the function $u$ given by Theorem~\ref{high-dimensional-delta-wings}
is entire.  In particular, there is an $(n-2)$-parameter family of entire translators $u:\RR^n\to \RR$,
no two of which are congruent to each other.
\end{corollary}

This is in sharp contrast to the situation in $\RR^2$: by the work of S.~J.~Wang and Spruck-Xiao,
the bowl soliton is the only entire translator $u:\RR^2\to\RR$.

\newcommand{\Ee}{\mathcal{E}}
\renewcommand{\aa}{\mathbf{a}}

\begin{proof}[Proof of Theorem~\ref{high-dimensional-delta-wings}]
Let $\Delta_n$ be the set of $n$-tuples $\aa=(a_1, a_2, \dots, a_n)$ of nonnegative
numbers such that $\sum a_i=1$.
Given $\aa=(a_1, \dots, a_n)$ in $\Delta_n$ and $\lambda>0$, consider the ellipsoidal region
\begin{equation}\label{n-ellipsoid}
  \Ee(\aa,\lambda) = \left\{x:  \sum_{i=1}^n a_i x_i^2 \le R^2 \right\},
\end{equation}
where $R>0$ is chosen so that if
\[
   u=u_{\aa,\lambda}:\Ee(\aa,\lambda)\to \RR
\]
is the bounded translator with boundary values $0$, then
\[
   u(0) = \lambda.
\]

Let $\kk = (k_1, \dots, k_n)$ be the principle curvatures of the graph of $u$ at the the maximum.
Thus $D^2u(0)$ is a diagonal matrix whose diagonal entries are $-k_1, \dots, -k_n$.
By Theorem~\ref{order-preserving-theorem} below,
 if $k_i=k_j$, then $a_i=a_j$, and thus $u$ is rotationally invariant about the $(n-2)$-dimensional
axis $\{x_i=x_j=0\}$:
\begin{equation}\label{curvature-equality-implies-symmetry}
   \text{$k_i=k_j$ implies $u_{\aa,\lambda}$ is rotationally symmetric about $\{x_i=x_j=0\}$}.
\end{equation}

Let 
\[
   F = F_n^\lambda: \Delta_n \to \Delta_n
\]
be the map that maps $\aa$ to $\kk$.

According to Theorem~\ref{long-domain-theorem}, 
if $a_i=0$, then $u$ is translation-invariant
in the $\ee_i$-direction, and thus $k_i=0$.
It follows that $F^\lambda_n$ maps each face of $\Delta$ to itself.
We can also conclude that $F^\lambda_n$ restricted to an $(m-1)$-dimensional face of $\Delta_n$ agrees
with $F^\lambda_m$.   Thus, for example,
\[
    F^\lambda_n(x_1,\dots, x_{n-1},0) = F^\lambda_{n-1}(x_1,\dots, x_{n-1})
\]
and
\[
    F^\lambda_n(x_1,\dots, x_{i-1},0,x_{i+1}, \dots x_n) = F^\lambda_{n-1}(x_1,\dots,x_{i-1},  x_{i+1}, \dots, x_n).
\]

Consequently, by elementary topology, $F^\lambda_n:\Delta_n\to\Delta_n$ is surjective.

Now fix a $\kk$ in $\Delta_n$.
For each $\lambda>0$, choose an $\aa=\aa(\lambda)$ in $\Delta_n$ such that $F^\lambda_n(\aa)=\kk$.

Now take a subsequential limit of 
\begin{equation}\label{approximators}
    u_{\aa(\lambda),\lambda}-u_{\aa(\lambda), \lambda}(0)
\end{equation}
 as $\lambda\to\infty$.

The result is a complete, properly embedded translator $u:\Omega \to \RR$.
Clearly $\max u=u(0)=0$, the Hessian $D^2u(0)$ has the specified form, and $u$ is an even
function of each coordinate (since the approximating functions~\eqref{approximators}
have that property.)

As already mentioned, if $a_i=0$, then $u_{\aa(\lambda),\lambda}$ is translation-invariant in the $x_i$-direction
and thus $D_iu_{\aa(\lambda),\lambda}\equiv 0$.
On the other hand, if $a_i>0$, then $D_iu_{\aa(\lambda)}>0$ wherever $x_i<0$ by the Alexandrov moving planes argument.
Thus, either way, we have
\[
   D_i u_{\aa(\lambda)} \ge 0\quad\text{wherever $x_i<0$}.
\]
Passing to the limit, we have
\begin{equation}\label{weak-alexandrov}
  D_iu \ge 0\quad\text{wherever $x_i<0$}.
\end{equation}
By differentiating the translator equation~\eqref{translator-equation} with respect to $x_i$, we see that $v=D_iu$
satisfies an linear elliptic PDE of the form
\[
  a_{jk}D_{jk}v + b_jD_jv = 0.
\]
Hence by~\eqref{weak-alexandrov} and the strong maximum principle, either $v\equiv 0$, in which case
$u$ is translation-invariant in the $\ee_i$-direction and so $k_i=0$, or else
\[
  D_iu (= v) >0 \quad\text{wherever $x_i<0$}.
\]
In the latter case, $D_iv(0)<0$ by the Hopf Boundary Point Lemma.
That, $k_i\ne 0$.  This completes the proof of Assertion~\eqref{alexandrov-item}.

Assertion~\eqref{axis-rotation-item} follows immediately from~\eqref{curvature-equality-implies-symmetry}.

To prove Assertion~\eqref{domain-item},
we may suppose (by relabelling the variables) that $k_1$ is the largest
of the principal curvatures:
\[
  k_1 \ge k_i \quad\text{for all $i$}.
\]

{\bf Case 1}: The entire $x_1$-axis lies in $\Omega$.
By Theorem~\ref{low-point-theorem} below,
\[
     u(\rho,0,\dots,0) = \min \{ u(x): |x|=\rho\}
\]
for each $\rho\ge 0$.  Thus $u$ is entire.

\newcommand{\bb}{\mathbf{b}}

{\bf Case 2}:
$\Omega$ does not contain the entire $x_1$-axis.  Then $\partial \Omega$ contains 
a point $\bb=(b,0, \dots,0)$.  By symmetry, we can assume that $b>0$.

Note by Assertion~\eqref{alexandrov-item} that if $x$ is in the domain of $u$, then so
is its projection to each coordinate hyperplane and therefore (iterating) so is its projection
to each coordinate axis.
Thus $\Omega$ lies in the region $\{x: x_1\le b\}$.
Now $(\partial \Omega)\times\RR$ is a minimal variety with respect to the Ilmanen metric
since it is the limit of translators.  Hence $\partial \Omega$ is a minimal variety with respect
to the Euclidean metric.
It lies on one side of the plane $\{x: x_1=b\}$ and touches it at $\bb$.
Hence by the maximum principle, $\partial \Omega$ contains all of that plane.
By symmetry, $\partial \Omega$ also contains the plane $\{x: x_1=(-b)\}$.
It follows that $\Omega=\{x: |x_1|<b\}$.

(If the last sentence is not clear, note 
that if $L$ is a line parallel to a coordinate axis, then $L\cap\Omega$ is connected
by Assertion~\eqref{alexandrov-item}.)

Note that $k_1>k_i$ for each $i\in \{2,3,\dots,n\}$.  For if $k_1=k_i$, then $u$
and therefore its domain would be rotationally symmetric about $\{x_1=x_i=0\}$,
and the slab $\{x: |x_1|<b\}$ does not have that symmetry.
\end{proof}

\section{Translating Graphs over Ellipsoidal Domains}

In this section, we prove the properties of translating graphs over
ellipsoids that were used in the proof of Theorem~\ref{high-dimensional-delta-wings}.
Throughout this section, 
\[
   u: \Ee\to\RR
\]
is a bounded translator with boundary values $0$, where
\[
  \Ee = \left\{x\in \RR^n: \sum a_i x_i^2 \le R^2\right\}.
\]
The $a_i$'s are nonnegative and not all $0$, and $R>0$.
We let $k_1, k_2,\dots,k_n$ be the principal curvatures at $x=0$, so that
$D^2u(0)$ is the diagonal matrix with diagonal entries $-k_1, -k_2,\dots, -k_n$.

We may assume that $a_i>0$ for each $i$, 
since if any $a_i=0$, then $u$ is translation-invariant
in that direction, and thus $u$ is given by a lower-dimensional example.

\begin{theorem}\label{same-sign-theorem}
Suppose $a_1>a_2$. 
Then
$
   (x_1D_2-x_2D_1) u
$
and $x_1x_2$ have the same sign at all points where $x_1x_2\ne 0$.
\end{theorem}

Of course if $a_1=a_2$, then (by uniqueness of the solution $u$), $u$ is rotationally
symmetric about the subspace $\{x_1=x_2=0\}$, so $(x_1D_2-x_2D_1) u\equiv 0$.

\begin{proof}
Let $M$ be the graph of $u$.  Recall that $M$ is stable in the Ilmanen metric $g$
(since vertical translations give a Jacobi Field that is everywhere positive.)
Consequently, if any Jacobi Field on a connected region $U$ in $M$ is nonnegative
on $\partial U$, then it is nonnegative on all of $U$, and if it is positive on some
portion of $\partial U$, then it is positive everywhere in the interior of $U$ by the strong maximum principle.

Let
\[
 f= (x_1D_2-x_2D_1) u.
\]
Note that $f$ is equivalent to the Jacobi Field coming from rotating $M$ about the $\{x_1=x_2=0\}$ plane.
Consider $f$ on the region $Q:=\{x\in \Ee: x_1\ge 0, \, x_2\ge 0\}$.

Clearly $f=0$ on the portions of $\partial Q$ where $x_1=0$ or $x_2=0$.
On the remaining portion of $\partial Q$, that is, on $\partial \Ee\cap \{x_1>0,\,x_2>0\}$,
we see that $f>0$ since $a_1>a_2$.  
Thus $f> 0$ everywhere in the interior of $Q$:  
\begin{equation}\label{positive-inside} 
\text{$f>0$ everywhere in $Q\setminus\{x_1x_2=0\}$}.
\end{equation}
By symmetry, it follows that $f>0$ wherever $x_1x_2>0$
and that $f<0$ wherever $x_1x_2<0$.
\end{proof}

\begin{theorem}\label{order-preserving-theorem}
If $a_1>a_2$, then $k_1>k_2$.
\end{theorem}

\begin{proof}
Let $R_\theta:\RR^n\to\RR^n$ denote rotation by $\theta$
about the $x_1x_2$-plane.  Thus
\[
  \frac{d}{d\theta}R_\theta(x) =  x_1\ee_2 - x_2\ee_1.
\]
For each $\theta$, the function $u\circ R_\theta$ is also a solution of the translator equation, so
\[
  f:= \left(\pdf{}{\theta}\right)_{\theta=0} u\circ R_\theta = (x_1D_2 - x_2D_1)u
\]
solves the linearization of the 
 translator equation~\eqref{translator-equation}.
In particular, 
\[
   a_{ij}D_{ij}f + b_i D_if = 0, 
\]
where $a_{ij}(0)=\delta_{ij}$ and $b_i(0)=0$.  Thus the lowest nonzero
polynomial $P(x)$ in the Taylor series for $f$ at $0$ is harmonic.
Since it has the same sign as the homogeneous harmonic polynomial $x_1x_2$,
in fact 
\begin{equation}\label{polynomial}
 \text{$P(x)=c\,x_1x_2$ for some constant $c>0$.}
\end{equation}
Now
\[
  u(x) = -\sum_i k_ix_i^2 + O(|x|^3)
\]
so
\[
 (x_1D_2-x_2D_1) u(x) = (k_1-k_2)2x_1x_2 + O(|x|^3).
\]
Hence $k_1-k_2>0$ by~\eqref{polynomial}.
\end{proof}

\begin{theorem}\label{low-point-theorem}
Extend $u$ to all of $\RR^n$ by setting $u(x)=0$ for $x\notin \Ee$.

Suppose that
\[
    k_1 = \max_{i\le n} k_i,
\]
Then for each $\rho>0$,
\[
   u(\rho \ee_1)  = \min \{ u(x): |x|=\rho\}.
\]
\end{theorem}

\begin{proof}
By Theorem~\ref{order-preserving-theorem}, 
\[
  r_1 \le r_i \quad \text{for all $i$}.
\]
We may assume that $\rho\ee_1$ is in $\Ee$, as otherwise
as otherwise the assertion is trivially true.  Hence the entire sphere $\{x: |x|=\rho\}$
is contained in $\Ee$.

Suppose
\[
  \min_{\partial \BB(0,\rho)} u
\]
occurs at the point $x$.  By symmetry, we can assume that $x_i\ge 0$ for each $i$.

Let $i>1$.  Define $\tilde x$ by
\begin{align*}
  \tilde x_1 &= \sqrt{x_1^2 + x_i^2},  \\
  \tilde x_i  &= 0, \\
  \tilde x_j &= x_j \quad\text{for $j\ne 1, i$}.
\end{align*}
Applying this with $i=2,3,\dots, n$, we see that the minimum is attained at $(\rho,0,0,\dots,0)$.
\end{proof}

\section{Ellipsoidal Slabs}\label{ellipsoidal-slabs-section}

\renewcommand{\aa}{\mathbf{a}}

As before, $\Delta^n$ is the set of $\aa=(a_1,\dots,a_n)$ where each $a_i\ge 0$ and $\sum_ia_i=1$.
Given $\aa\in \Delta^n$ and $R>0$, we let
\[
 \Ee=\Ee(\aa,R)= \left\{x\in \RR^n: \sum_i a_ix_i^2 \le R^2 \right\}.
\]

For $\aa\in \Delta^n$ and $b>\pi/2$, we let
\[
   u=u_{\aa,b,R}: \Ee(\aa,R)\times [-b,b] \to \RR
\]
be the bounded translator with boundary values $0$.
Let $k_i = -D_{ii}u_{\aa,b,R}(0)$.
The theorems in this section describe properties of such $u$.
In Section~\ref{higher-slabs-section}, we will let $R\to\infty$ to get complete translators defined over 
the slab $\RR^n\times(-b,b)$.

\begin{theorem}\label{lionel} For $u=u_{\aa,b,R}$,
\begin{enumerate}[\upshape (1)]
\item\label{lionel-long-item} If $a_i=0$, then $D_iu\equiv 0$.
\item\label{lionel-first-alexandrov-item} If $a_i>0$, then $D_iu(x)$ and $x_i$ have opposite signs at all points
in the interior of the domain, and $D_{ii}u(0)<0$.
\item\label{lionel-second-alexandrov-item} 
    $D_{n+1}u(x)$ and $x_{n+1}$ have opposite signs at all points in the interior
    of the domain, and $D_{n+1,n+1}u(0)<0$.
\item\label{lionel-rotational-equality-item} If $1\le i,j\le n$ and if $a_i=a_j$, then
\[
    ( x_i D_j - x_j D_i) u \equiv 0.
\]
\item\label{lionel-rotational-inequality-item}  If $1\le i, j \le n$ and $a_j<a_i$, then
$
k_i>k_j
$,
and
\[
 ( x_i D_j - x_j D_i) u
\]
 and $x_ix_j$
have the same sign at all points in the interior of $\Ee$.
\end{enumerate}
\end{theorem}
(The theorem does not assert any relationship between $k_i$ and $k_{n+1}$.)

\begin{proof}
Assertion~\eqref{lionel-long-item} follows immediately from Theorem~\ref{long-domain-theorem}.
Assertions~\eqref{lionel-first-alexandrov-item} and~\eqref{lionel-second-alexandrov-item}
follow from the Alexandrov moving plane argument.
The proofs of Assertions~\eqref{lionel-rotational-equality-item}
and~\eqref{lionel-rotational-inequality-item} 
are almost identical to the proofs of Theorems~\ref{same-sign-theorem}
and~\ref{order-preserving-theorem}.
\end{proof}

\begin{theorem}\label{monotone-along-edge}
If $\eta$ is the inward pointing unit normal on $\partial \Ee(\aa,R)$ and if
\[
    (x_1,x_2,\dots,x_n)\in \partial \Ee(\aa,R),
\]
 then
\[
  s \in [-b,b] \mapsto  \eta\cdot\nabla u(x_1,x_2,\dots,x_n,s) = |Du(x_1,x_2,\dots,x_n,s)|
\]
is a decreasing function of $|s|$.
\end{theorem}

\begin{proof}
Since $D_{n+1}u=0$ on $(\partial\Ee)\times[0,b]$ and since $D_{n+1}u<0$
on $\operatorname{interior}(\Ee) \times (0,b]$, it follows that
\[
   D_\eta D_{n+1}u \le 0 \quad \text{on $(\partial \Ee)\times (0,b]$}.
\]
Hence
\[
  D_{n+1} D_\eta u \le 0 \quad \text{on $(\partial \Ee)\times [0,b]$}.
\]
Likewise,
\[
  D_{n+1} D_\eta u \ge 0 \quad \text{on $(\partial \Ee)\times [-b,0]$}.
\]
\end{proof}

\begin{theorem}\label{w-bound-theorem}
Let $\aa\in \Delta^n$, $R\ge 1$, and let
\[
   u: \Ee(\aa,R)\times[-b,b]\to\RR
\]
be the translator with boundary values $0$.
Let
\[
   w(x) = (b - |x_{n+1}|) \sqrt{1+|Du(x)|^2}.
\]
Then
\[
   \max w \le C,
\]
where $C=C(n,b)<\infty$ depends only on $n$ and $b$.
\end{theorem}

\begin{proof}
Let $\delta(x)= (b-|x_{n+1}|)$.

Suppose the theorem is false.  Then there is a sequence
of translators
\[
   u_k: \Ee(\aa(k),R(k))\times [-b,b]\to \RR
\]
and a sequence points $p(k)$ such that
\begin{equation}\label{blowing-up}
   \delta(p(k))\sqrt{1+|Du_k(p(k))|^2} \to \infty.
\end{equation}

By passing to a subsequence, we can assume that $R(k)$ converges
to a limit $R$ in $[1,\infty]$.  The case $R<\infty$ is straightforward,
so we assume that $R(k)\to \infty$.

We may suppose that $p(k)$ has been chosen to maximize the
left side of~\eqref{blowing-up}.
By symmetry, we may suppose that each coordinate of $p(k)$ is $\ge 0$:
\begin{equation}\label{in-octant}
  p(k) \in [0,\infty)^{n+1}.
\end{equation}
By Theorem~\ref{spruck-xiao-gradient-bound}, 
   $p(k)$ occurs on the boundary of $\Ee(\aa(k),R(k))\times [0,b]$.

Note that $p(k)$ does not lie on $\Ee(\aa(k),R(k))\times \{b\}$ since $\delta\equiv 0$ there.
Also, $p(k)$ does not lie on $(\partial \Ee)\times (0,b]$ by Theorem~\ref{monotone-along-edge}.
Thus
\[
  p(k) \in \Ee(\aa(k),R(k))\times \{0\}.
\]
Let $\tilde p(k) = (p(k), u_k(p(k)))$.

{\bf Case 1}: $\dist(\tilde p(k), \partial M(k))$ is bounded above.

Translate $M(k)$ by $\tilde p(k)$ to get a surface $M'(k)$.
By passing to a subsequence, $M'(k)$ converges to a limit translator $M$.
Note that $\partial M$ is nonempty, horizontal, and has corners.
It follows (see Theorem~\ref{upper-halfspace-compactness}) 
that the convergence of $M'(k)$ to $M$ is smooth except at the corners of $\partial M'$.

Note that the tangent plane to $M$ at $0$ is vertical.  
By the strong maximum principle (if $0\notin \partial M$) or the boundary maximum principle
(if $0\in \partial M$), the tangent plane to $M$ is vertical at every point.   In particular, if $q\in\partial M$,
then $q+t\ee_{n+2}\in M$ for all $t>0$.  But since $\partial M$ has corners, this contradicts the smoothness
of $M$.   Thus Case 1 cannot occur.

{\bf Case 2}: $\dist(\tilde p(k),\partial M(k))$ is unbounded.  By passing to a subsequence,
we can assume that the distance tends to infinity.

Translate $M(k)$  by $-p(k)$ to get $M'(k)$.

Note that $0\in M'(k)$.  Furthermore, by Theorem~\ref{lionel},
\begin{equation}\label{disjoint-region}
\text{$M(k)$ is disjoint from $(0,\infty)^n\times \RR \times (0,\infty)$}
\end{equation}
By passing to a subsequence, we can assume that the $M'(k)$ converge smoothly to a complete,
properly embedded translator $M$ in $\RR^n\times (-b,b)\times \RR$.  (In fact, $M$ is area-minimizing
for the Ilmanen metric.)

Note that $0\in M$ and that the tangent plane to $M$ at $0$ is vertical.  Thus by the maximum
principle, $M$ is vertical everywhere.  That is,
\[
   M = \Sigma\times \RR
\]
for a smooth Euclidean minimal hypersurface $\Sigma\subset \RR^n\times (-b,b)$.
Also, by~\eqref{disjoint-region},
\begin{equation}\label{more-disjoint-region}
\text{$\Sigma$ is disjoint from $(0,\infty)^n\times \RR$.}
\end{equation}
However, using moving catenoids shows that no such surface $\Sigma$ exists.

\newcommand{\cat}{\operatorname{Cat}}

(Here is the moving catenoid argument.  Let $\cat$ be an $n$-dimensional catenoid, i.e., a
minimal hypersurface of revolution about the $x_{n+1}$-axis, bounded by spheres in the plane $x_{n+1}=b$
and the plane $x_{n+1}=-b$.  Let $\cat(s)$ be $\cat$ translated by
\[ 
     s( \ee_1 + \dots + \ee_n).
\]
For $s>0$ very large $\cat(s)$ is disjoint from $M$ by~\eqref{more-disjoint-region}.
Thus $\cat(s)$ is disjoint from $M$ for all $s$, since otherwise the strong maximum
principle would be violated at the first point of contact.
  But this is a contradiction since $0\in M$ and since
 there is an $s$ from which $0\in \cat(s)$.)
 \end{proof}

\section{Higher Dimensional $\Delta$-Wings in Slabs of Prescribed Width}\label{higher-slabs-section}

As in the previous section, we let
\[
  u_{\aa,b,R}: \Ee(\aa,R)\times[-b,b]\subset \RR^{n+1} \to\RR
\]
be the translator with boundary values $0$, and let
\[
   k_i = -D_{ii}u(0).
\]

Now define
\[
  F_{b,R}: \Delta^n \to \Delta^n
\]
as follows.  

Let
\[
  F_{b,R}(\aa) = \frac{(k_1, k_2, \dots, k_n)}{\sum_{i=1}^nk_i}.
\]

As before, $F_{b,R}$ is continuous, and it maps each face of $\Delta^n$
to itself.

Thus $F_{b,R}$ is surjective.

Fix $b>\pi/2$ and $\mathbf{k}\in \Delta^n$.

By surjectivity, for each $R>0$, there is an $\aa(R)\in \Delta_n$ such that
\[
   F_{b,R}(\aa(R)) = \mathbf{k}.
\]

Now every sequence of $R\to\infty$ has a subsequence $R_i\to\infty$ such that
\[
   u_{\aa(R_i), b, R_i} - u_{\aa(R_i), b, R_i}(0)
\]
converges smoothly to a limit
\[
   u: \RR^n\times (-b,b)\to \RR.
\]
Note the slope bound Theorem~\ref{w-bound-theorem} 
implies that the domain of $u$ is the entire slab $\RR^n\times (-b,b)$
rather than some open subset of it.

We have proved

\begin{theorem}\label{higher-slab-existence-theorem}
Let $0\le k_1\le k_2 \le \dots \le k_n$ with $\sum_ik_i=1$ and let $b>\pi/2$.
Then there is a complete translator
\[
   u: \RR^n\times (-b,b)\to \RR
\]
such that 
\begin{align*}
   u(0)  &= 0, \\
   u(x) &< 0 \quad\text{for $x\ne 0$,}
\end{align*}
and such that $-D^2u(0)$ is a positive definite diagonal matrix with
diagonal entries $\kappa_1,\dots, \kappa_{n+1}$ where
\[
   (\kappa_1,\dots, \kappa_n)
\]
is a scalar multiple of $(k_1,\dots, k_n)$.

If  $k_i=k_j$ (where $1\le i< j< n$), then $(x_iD_j-x_jD_i)u\equiv 0$.
\end{theorem}

\begin{corollary}
For each width $b>\pi/2$, 
 there is an $(n-1)$-parameter family of such $\Delta$-wings $u$, 
no two of which are congruent to each other.
\end{corollary}

In the special case $k_1=k_2=\dots=k_n$, then $u$ is rotationally invariant about
the $x_{n+1}$ axis (since $(x_iD_j-x_jD_i)u\equiv 0$).  
In that case, we have the following uniqueness theorem:

\begin{theorem}\label{higher-uniqueness} Let $b>\pi/2$.  Then there is a unique complete, proper
translator 
\[
   u: \RR^n\times (-b,b)\to\RR
\]
with the following properties:
\begin{enumerate}[\upshape (1)]
\item (Rotational invariance) $u$ is rotationally invariant about the $x_{n+1}$ axis:
\[
  u(x_1,x_2, \dots, x_n, x_{n+1}) \equiv u( \sqrt{x_1^2 + \dots + x_n^2},0,0,\dots, x_{n+1}).
\]
\item  (Slope bound) For every $\beta<b$,
\[
   \sup_{\RR^n\times [-\beta,\beta]} |Du| < \infty.
\]
\end{enumerate}
\end{theorem}

Existence is a special case of Theorem~\ref{higher-slab-existence-theorem}.
The proof of uniqueness is essentially the same as the proof of
 Theorem~\ref{symmetric-uniqueness}, so we omit it.
The arguments in the proof of Theorem~\ref{symmetric-uniqueness}  
are very two dimensional, but if we mod out by the symmetry,
then we are dealing with a function of two variables.  Specifically, we let 
\begin{align*}
   &U: \RR\times (-b,b) \to \RR, \\
   &U(x,y) = u(x,0,0,\dots,0,y).
\end{align*}

Existence (but not uniqueness) of rotationally invariant examples was proved in
a different way, using barriers, by Bourni et al.  The barriers show that their
examples satisfy the slope bound hypothesis in Theorem~\ref{higher-uniqueness}.

In general, we do not know whether the surfaces in
 Theorem~\ref{higher-slab-existence-theorem} are convex.
However, if a complete translator has no more than two distinct principal curvatures
at each point, then it is convex \cite[Theorem 3.1]{bourni-et-al}.  Hence in the special case $k_1=\dots=k_n$,
the surface is convex.


\section{Appendix: Compactness Theorems}\label{appendix}

\begin{theorem}\label{compactness-theorem}
For $k=1, 2, \dots$, let $\Omega_k$ be a convex open subset of $\RR^n$
and let $u_k:\Omega_k\to \RR$ be a smooth translator.
Let $M_k$ be the graph of $u_k$.   
Suppose that $W$ is a connected open subset of $\RR^n$ such that for each $k$,
\[
    W\times (-k,k)
\]
does not contain any of the boundary of $M_k$.

Then, after passing to a subsequence, $M_k\cap (W\times \RR)$ 
converges weakly  in $W\times \RR$ to a translator $M$ that is $g$-area-minimizing.
Furthermore, if $S$ is a connected component of $M$, then either
\begin{enumerate}
\item $S$ is the graph of a smooth function over an open subset of $W$ and the convergence
to $S$ is smooth, or
\item  $S=\Sigma \times \RR$, where $\Sigma$ is a variety in $W$ that is minimal
   with respect to the Euclidean metric on $\RR^n$.  The singular set of $\Sigma$ has Hausdorff
   dimension at most $n-7$.
\end{enumerate}
\end{theorem}

\begin{proof}
Since $\Omega_k$ is convex and since
$M_k$ and its vertical translates form a $g$-minimal foliation of $\Omega_k\times \RR$,
standard arguments (cf.~\cite{morgan-book}*{\S6.2}) show that $M_k$ is $g$-area-minimizing as an integral current,
or even as a mod $2$ flat chain.

Thus the standard compactness theorem (cf.~\cite{simon-gmt-book}*{\S34.5}) 
gives subsequential convergence (in the local flat topology)
to a $g$-area-minimizing hypersurface $M$ (with no boundary in $W\times I$).
Also, standard arguments show that the support of $M_k$ converges to the support of $M$.
Hence we will not make a distinguish here
between the flat chain and its support.

For notational simplicity, let us assume that $M$ is connected.
Clearly, each vertical line intersects $M$ in a connected set.

{\bf Case 1}: $M$ contains a vertical segment of some length $\eps>0$.
 Let $M(s)$ be the result of translating $M$ vertically
by a distance $s$, where $0<s<\eps$. 
Then by the strong maximum principle of L. Simon~\cite{simon-maximum}, 
$M=M(s)$.   Since this is true for all $s$ with $0<s<\eps$, it follows that $M=\Sigma\times\RR$ for some $\Sigma$.
Since $\Sigma\times\RR$ is $g$-area-minimizing, its singular set has Hausdorff dimension
at most $(n+1)-7$, and therefore the singular set of $\Sigma$ has Hausdorff dimension
at most $n-7$.  Since $\Sigma\times \RR$ is $g$-minimal, $\Sigma$ must be minimal
with respect to the Euclidean metric.

{\bf Case 2}: $M$ contains no vertical segment.
Then $M$ is the graph of a continuous function $u$ whose domain is an open subset of $W$.
Let $\overline{\BB(p,r)}$ be a closed ball in the domain of $u$.
Then $\overline{\BB(p,r)}$ is contained in the domain of $u_k$ for large $k$, and $u_k$
converges uniformly to $u$ on $\overline{\BB(p,r)}$. (The uniform convergence follows
from monotonicity.)   By Theorem~5.2 of~\cite{evans-spruck-iii}
(rediscovered in \cite{colding-minicozzi-sharp}*{Theorem~1}), 
the convergence is smooth on $\BB(p,r)$.
\end{proof}

\begin{theorem}\label{upper-halfspace-compactness}
For $k=1, 2, \dots$, let $\Omega_k$ be a convex open subset of $\RR^n$
such that the $\Omega_k$ converge to an open set $\Omega$. Let $u_k:\overline{\Omega_k}\to\RR$ be a 
translator with boundary values $0$, and
let $M_k$ be the graph of $u_k$.
Then, after passing to a subsequence, the $M_k$ converge smoothly in $\RR^n\times (0,\infty)$
to a smooth translator $M$.   If $S$ is a connected component of $M$, then
either
\begin{enumerate}
\item $S$ is the graph of a smooth function whose domain is an open subset of $\Omega$, or
\item $S=\Sigma\times [0,\infty)$, where $\Sigma$ is an $(n-1)$-dimensional affine plane in $\RR^n$.
\end{enumerate}
Furthermore, $M$ is a smooth manifold-with-boundary in a neighborhood of every point
of $\partial M$ where $\partial M$ is smooth, and the convergence of $M_k$ to $M$ is smooth
up the boundary wherever the convergence of $\partial M_k$ to $\partial M$ is smooth.
\end{theorem}

\begin{proof}
Case 1: $S$ contains a point $p$ in $\partial \Omega\times (0,\infty)$.
Let $\Sigma$ be the connected component of $\partial \Omega$ such that $p\in \Sigma\times (0,\infty)$.
Then by the strong maximum principle (\cite{solomon-white} or \cite{simon-maximum}
or~\cite{white-controlling}*{Theorem~7.3}), $M$ contains all of $\Sigma\times (0,\infty)$, and therefore
(since $\Omega$ is convex) $\Sigma$ must be a plane.   It follows that the convergence to $S$ is smooth
in $\RR^n\times (0,\infty)$.

Case 2: $S$ contains no point in $\partial \Omega\times \RR$.  
That is, $S$ is contained in $\Omega\times \RR$.
Indeed, $S$ is contained in $\Omega\times[0,\infty)$, 
so it cannot be translation-invariant in the vertical direction.
Thus by Theorem~\ref{compactness-theorem}, 
$S$ is a smooth graph over an open subset of $\Omega$, and the convergence
to $S$ is smooth.

The assertions about boundary behavior follow, for example, from the Hardt-Simon Boundary
Regularity Theorem~\cite{hardt-simon-boundary}.  
(Note that the tangent cone to $M$ at a regular point of $\partial M$ is,
after a rotation of $\RR^n$, a cone in $\RR^{n-1}\times [0,\infty)\times[0,\infty)$ whose
boundary is the plane $\RR^{n-1}\times\{0\}\times\{0\}$.)
\end{proof}


\end{document}